\documentclass{mod}
\usepackage{url}
\usepackage{setspace}
\usepackage{scrextend}
\usepackage{cite}
\usepackage[shortlabels]{enumitem}
\usepackage{hyperref}
\usepackage{caption}
\usepackage{subcaption}
\usepackage{float}
\usepackage{tikz}
\usepackage{tikz-cd}
\usepackage{amsmath}
\usepackage{amsthm}
\usepackage{amssymb}

\numberwithin{equation}{section}
\renewcommand{\theequation}{\thesection.\arabic{equation}}

\DeclareMathAlphabet{\mathbbm}{U}{bbm}{m}{n}

\usepackage{mathtools}
\usepackage{mathrsfs}

\usepackage[all]{xy}
\usepackage[toc,page]{appendix}
\usepackage{titlesec}
\usepackage{upgreek}

\usepackage{tocloft}

\usepackage{tocbasic}
\DeclareTOCStyleEntry[
beforeskip=.1em plus 1pt,
entryformat=\normalfont ,
pagenumberformat=\bfseries
]{tocline}{section}
\usepackage{etoolbox}
\patchcmd{\thebibliography}{%
	\section*{\refname}\@mkboth{\MakeUppercase\refname}{\MakeUppercase\refname}}{%
	\section*{\refname}}{}{}

\usepackage{ stmaryrd }

\newcounter{mainthm}

\newcounter{mainconj}

\newtheorem{thm}{Theorem}[section]

\theoremstyle{plain}
\newtheorem{lem}[thm]{Lemma}
\newtheorem{prop}[thm]{Proposition}

\newtheorem{defn-thm}[thm]{Definition-Theorem}

\newtheorem{defn-lem}[thm]{Definition-Lemma}
\newtheorem{construction}[thm]{Construction}

\newtheorem{defn}[thm]{Definition}

\theoremstyle{definition}

\newtheoremstyle{rmk}
{5pt}
{5pt}
{}
{}
{\itshape}
{}
{.5em}
{}

\newtheorem{rmk}[thm]{Remark}
\newtheorem{ex}[thm]{Example}

\newtheoremstyle{note}
{8pt}
{5pt}
{\itshape}
{10pt}
{\bfseries}
{}
{.5em}
{}

\theoremstyle{note}

\setlist[description]{font=
	\normalfont
	\itshape
	\space}

\usepackage{titletoc}

\titlecontents{section}
[0.72em]
{\scriptsize\bfseries\vspace{3pt}}%
{\thecontentslabel.\enspace}
{}
{\titlerule*[0.38pc]{.}\contentspage}%

\titlecontents{subsection}
[0em]
{\scriptsize \vspace{0pt}}%
{\qquad \quad \thecontentslabel.\enspace}
{}
{\titlerule*[0.5pc]{.}\contentspage}%

\titlecontents{subsubsection}
[0em]
{\footnotesize \vspace{1pt}}%
{\qquad \qquad \thecontentslabel.\enspace}
{}
{\titlerule*[0.5pc]{.}\contentspage}%

\titleformat{\subsubsection}[runin]{
	\bfseries
	\itshape\normalsize}{(\thesubsubsection) \ }{0em}{}[\mbox{ . } ]
\titlespacing{\subsubsection}
{0pt}
{2.25ex plus 1ex minus .2ex}
{0pt}

\newcommand{\trop}{ {\mathfrak{trop}} }
\DeclareMathOperator{\val}{\mathsf{v}}

\def\bar{\overline}
\def\hat{\widehat}

\def\^{\wedge}

\def\C{\mathbb{C}}
\def\H{\mathcal{H}}
\def\P{\mathbb{P}}

\def\R{\mathbb{R}}
\def\Z{\mathbb{Z}}
\def\T{\mathbb{T}}
\def\X{\mathbb{X}}

\def\cH{\mathcal{H}}

\def\cH{\mathcal{H}}

\def\a{\alpha}

\def\b{\beta}
\def\g{\gamma}

\def\l{\lambda}

\def\L{\Lambda}

\def\r{\rho}

\def\w{\omega}

\begin{document}
	\setlength{\parindent}{15pt}	\setlength{\parskip}{0em}

	\title{Family Floer SYZ mirror algorithm for the Grassmannian $Gr(2,4)$}
	
	\renewcommand{\thanks}[1]{\footnote{#1}}	
	
	\author[Z. Yu and H. Yuan]{Zekai Yu \thanks{\tiny Tsinghua University (Qiuzhen College), Beijing 100084, China. \quad Email: \texttt{yuzk23@mails.tsinghua.edu.cn}} and Hang Yuan \thanks{\tiny  Beijing Institute of Mathematical Sciences and Applications, Beijing 101408, China. \quad Email: \texttt{yuanhang@bimsa.cn}}}

	\begin{abstract} {\sc Abstract:}
		We give an explicit non-archimedean SYZ construction for the Landau-Ginzburg mirror of 
		$Gr(2,4)$. This work is complementary to the approach of Hong-Kim-Lau \cite{hong2023immersed} 
		to SYZ mirror symmetry for Grassmannians, while we focus on a more concrete fibration-level realization of the SYZ picture.
		Starting from a Lagrangian fibration on the A-side, we explicitly construct a non-archimedean analytic mirror fibration inside 
		the Berkovich analytification of the Langlands dual Grassmannian on the B-side. We show that the two fibrations 
		have identical smooth and singular loci and induce the same integral affine structure on 
		the smooth locus. Moreover, the natural disk-counting Landau-Ginzburg superpotential 
		agrees with the Marsh-Rietsch superpotential.
		While the construction is guided by the family Floer viewpoint, the proof proceeds mainly through explicit geometric constructions and does not rely on Floer-theoretic arguments. Thus, the Langlands-dual mirror and its superpotential are realized explicitly within a single framework, providing concrete geometric evidence for the SYZ principle.
	\end{abstract}
	
	\maketitle
	\tableofcontents

	\section{Introduction}
	
	A remarkable feature of mirror symmetry for Grassmannians is that the mirror is naturally visible inside the Langlands dual Grassmannian.  More precisely, for the Grassmannian
	\[
	X= Gr(k,n)
	\]
	of $k$-dimensional subspaces in $\mathbb C^n$,
	Rietsch’s Lie-theoretic mirror construction \cite{rietsch2008mirror}, later written explicitly in Pl\"ucker coordinates by Marsh-Rietsch \cite{marsh2020b}, realizes the Landau-Ginzburg (LG) mirror as a pair
	$(Y, W)$
	where $Y$ is the complement of a distinguished anti-canonical divisor in the Langlands dual Grassmannian 
	$Gr(n-k,n)$, and $W$ is the mirror superpotential expressed in terms of Pl\"ucker coordinates. This construction refines the earlier physical
	mirrors for Grassmannians of Eguchi-Hori-Xiong \cite{EguchiHoriXiong1997} and Hori-Vafa \cite{hori2002mirror}.

	This Lie-theoretic mirror for Grassmannians is expected to be compatible with the Strominger-Yau-Zaslow picture \cite{SYZ}. In this perspective, the mirror should arise, at least heuristically, by dualizing a Lagrangian torus fibration on the Grassmannian. More precisely, if $X$ admits a Lagrangian torus fibration
	$\pi:X\to B$,
	then the mirror space $Y$ is expected to be constructed as the total space of the dual torus fibration
	$\pi^\vee:Y\to B$,
	as illustrated schematically by
	\[
	\xymatrix{
		X \ar[dr]_{\pi} & & Y \ar[dl]^{\pi^\vee} \\
		& B &
	}
	\]
	Moreover, the Landau-Ginzburg superpotential $W$ on the mirror $Y$ is expected to encode counts of Maslov index two holomorphic disks bounded by the Lagrangian torus fibers.

	For Grassmannians and flag varieties, Gelfand-Cetlin systems provide Lagrangian fibrations over polytopes, with singular fibers appearing over the boundary of the polytope. 
	Nishinou-Nohara-Ueda \cite{nishinou2010toric} computed the disk potential of a regular Gelfand-Cetlin fiber in partial flag manifolds, such as $Gr(2,n)$, and showed that it agrees with the Hori-Vafa's prediction.
	Moreover, the work of Hong, Kim, and Lau \cite{hong2023immersed} argues that for Grassmannians $Gr(2,n)$, studying the Floer theory of certain specific singular Lagrangian submanifolds may retrieve a large amount of information in the Langlands dual Grassmannian and the corresponding LG superpotential in the works of Rietsch.
	Their approach relies on replacing the singular fiber by immersed Lagrangians and then gluing the Maurer-Cartan deformation spaces of charts appropriately, while a corresponding mirror fibration is not constructed in their work.

	Now, it is natural to ask the following questions.
	
	\begin{enumerate}[(i)]
		\itemsep 3pt
		\item \textit{Can the Langlands dual Grassmannian be recovered from the SYZ picture, as a geometric space obtained by dualizing a Lagrangian torus fibration on the original Grassmannian?}
		\item \textit{Can the Marsh-Rietsch superpotential be recovered from the counts of Maslov index two holomorphic disks associated with corresponding Lagrangian torus fibers?}
		
		\item \textit{Finally, can these two problems be treated simultaneously so that both emerge from one coherent framework?}
	\end{enumerate}

	In this paper, we give further evidence for the SYZ philosophy of Grassmannians. Building on our previous works \cite{Yuan_I_FamilyFloer, Yuan_local_SYZ,Yuan_conifold,Yuan_A_n}, our SYZ mirror construction produces not only the desired Landau-Ginzburg model, but also a mirror fibration satisfying natural matching conditions on the base space $B$.

	Let $Gr(k,n)_{\Bbbk}$ denote the Grassmannian over a fixed ground field $\Bbbk$.
	Denote by $\Lambda=\mathbb C((T^{\mathbb R}))$ the \textit{Novikov field} consisting of formal power series $\sum_{i=0}^\infty a_i T^{\lambda_i}$ with $a_i\in\mathbb C$ and $\lambda_i \nearrow \infty$.
	This is a non-archimedean field, so one can consider the Berkovich analytification $Gr(k,n)_{\Lambda}^{\mathrm{an}}$ of the variety $Gr(k,n)_{\Lambda}$; see \cite{Berkovich_2012spectral,Berkovich1993etale}.
	Intuitively, just as the complex analytic topology on $Gr(k,n)_{\mathbb C}$ refines the Zariski topology, 
	one can view $Gr(k,n)_{\Lambda}^{\mathrm{an}}$ as a topological space that refines the Zariski topology on $Gr(k,n)_{\Lambda}$.
	Note that
	$Gr(k,n)_{\Bbbk}$ has dimension $N:=k(n-k)$ and the Pl\"ucker coordinates give an embedding into a projective space.
	For the sake of concreteness and explicitness, we focus on the special case $k=2$ and $n=4$, and leave the general case to future work. 
	There is a distinguished anti-canonical divisor $D_{ac}$ in $Gr(2,4)$ given by the vanishings of the four cyclic Pl\"ucker coordinates $Z_{12}, Z_{23}, Z_{34}, Z_{14}$.

	Our main result is as follows:

	\begin{thm}
		\label{main_thm}
		Define $X=Gr(2,4)_{\mathbb C}\setminus D_{ac}$ and $Y=Gr(2,4)_{\Lambda}^{\mathrm{an}} \setminus D_{ac}$.
		There exists an analytic open subset $\mathcal Y\subset Y$, a Lagrangian fibration $\pi:X \to B$ and an analytically continuous map $\pi^\vee: \mathcal Y\to B$ for some base manifold $B$ such that
		
		\begin{itemize}
			\itemsep 2pt
			\item The smooth and singular loci of $\pi$ and $\pi^\vee$ coincide identically.
			\item The two integral affine structures induced from $\pi$ and $\pi^\vee$ respectively coincide identically.
		\end{itemize}
		
		\noindent
		Both $\pi$ and $\pi^\vee$ can be written down explicitly. Moreover, there is a natural LG superpotential obtained from counting disks intersecting $D_{ac}$ (see \eqref{eq_tilde_W}), which coincides with the Marsh-Rietsch superpotential in \cite{marsh2020b}.
	\end{thm}

	Despite the term ``Family Floer'' in the title, the statement of Theorem \ref{main_thm} is completely free of Floer theory, and most of the work consists of explicit constructions.
	We wish our main theorem to be accessible not only to symplectic geometers but also to a broader audience. Some steps in our mirror construction may appear unmotivated or look ad hoc, but the motivations firmly stem from a systematic Floer-theoretic framework developed in \cite{Yuan_I_FamilyFloer}.
	To sum up, regardless of the motivations, the proof of Theorem~\ref{main_thm} is carried out largely through explicit and direct constructions and is free of Floer-theoretic arguments.
	
	A brief outline of our mirror construction approach is as follows.
	By the Arnold-Liouville theorem, the smooth locus of the Lagrangian fibration, $\pi_0 : X_0 \to B_0$, can be locally identified with the standard complex logarithm map $\mathrm{Log} : (\mathbb C^*)^n \to \mathbb R^n$ via action-angle coordinates. More precisely, let $\{\chi_i : U_i \to V_i\}$ be an atlas for the integral affine structure on $B_0$. Then, each local fibration $\pi_0^{-1}(U_i)$ is symplectically identified with $\mathrm{Log}^{-1}(V_i)$, and globally
	\[
	X_0 \cong \bigcup_i \mathrm{Log}^{-1}(V_i) \big/ \sim_{SG}
	\]
	where $\sim_{SG}$ is a gluing in symplectic geometry compatible with the local fibration maps.
	The family Floer mirror construction asserts that the dual local fibrations are first given by \(\trop^{-1}(V_i) \to V_i\), and then that the Fukaya $A_\infty$ algebras associated to those smooth fibers determine a natural gluing
	\[
	X_0^\vee := \bigcup_i \trop^{-1}(V_i) \big/ \sim_{NA},
	\]
	where $\sim_{NA}$ is a gluing in the category of Berkovich analytic spaces. By definition, these $A_\infty$ algebras depend not only on the smooth locus $X_0$ but on the ambient symplectic manifold $X$; in particular, to some extent the gluing $\sim_{NA}$ encodes information about the singular Lagrangian fibers. Moreover, we proved in \cite{Yuan_I_FamilyFloer} that the analytic space $X_0^\vee$ is unique up to isomorphism.
	
	Nevertheless, despite this general theory for the gluing $\sim_{NA}$, in the presence of sufficient symmetries it may happen that the atlas has finitely many charts and that the gluing $\sim_{NA}$ can be described explicitly. This is precisely the situation of Theorem \ref{main_thm} where no prior knowledge of Floer theory and Fukaya's $A_\infty$ algebras is required.

	
	A short proof roadmap is as follows:
	
	\begin{enumerate}
		\itemsep 4pt
		\item Construct the A-side Lagrangian fibration and identify its singular locus (\S \ref{s_lag_fib}). 
		\item Compute disk-class local systems and monodromy (\S \ref{s_local_system_topo}). 
		\item Use disk areas to construct the integral affine atlas (\S \ref{s_int_aff_at}). 
		\item Build the mirror analytic space by two affinoid charts and the wall-crossing gluing (\S \ref{s_family_floer_mirror}). 
		\item Embed this mirror into the known Grassmannian mirror and construct the B-side fibration (\S \ref{s_explicit_mirror_fibration}). 
		\item Compare smooth loci and affine structures (\S \ref{s_explicit_mirror_fibration}). 
	\end{enumerate}
	
	\section{A-side} \label{Aside}
	
	Let $Gr(2,4)$ be the complex Grassmannian of 2-dimensional subspaces in $\mathbb C^4$. It is of complex dimension 4 and embeds into $\mathbb {CP}^5$ with Pl\"ucker coordinates $Z_{ij}$ ($1\leqslant i<j \leqslant 4$).
	To emphasize the frozen variables, we further set
	$
	g_i=g_{i,i+1}=Z_{i,i+1}
	$ for the frozen variables, while both notations may be used according to the context. Note that
	\[
	Z_{13}Z_{24}=g_{12}g_{34}+g_{23}g_{14}
	\]
	
	Set $D_{ij}=\{Z_{ij}=0\}$, and the distinguished anti-canonical divisor is
	\[
	D_{ac}=D_{12}\cup D_{23}\cup D_{34} \cup D_{14}
	\]
	We set $\mathbb X=Gr(2,4)$ and define 
	\[ X=\mathbb X \setminus D_{ac}
	\]
	to be the complement of the divisor $D_{ac}$.
	Through the {Pl\"ucker embedding}, the Fubini-Study form on $\mathbb P^5$ induces a symplectic form $\omega$ on $\mathbb X$.
	We identify $S^1$ with $\mathbb R/\mathbb Z$.
	For each $1\leqslant i\leqslant 4$, we consider the natural $S^1$-action $\theta_i:S^1\times \mathbb X\to \mathbb X$ on $\mathbb X$ defined by
	\begin{align}
		\label{Ham_S_1_theta_1234}
		\big( e^{2\pi\mathbf it} \ , \ Z_{jk} \big) &\mapsto 
		\begin{cases}
			e^{2\pi \mathbf i t}Z_{jk}& \text{, if } j=i\text{ or } k=i \\
			Z_{jk}& \text{, else}
		\end{cases}
	\end{align}
	Notice that $\theta_1,\theta_2,\theta_3,\theta_4$ are not independent since an overall phase rotation by $2\pi$ is trivial on the projective space. Using the standard Fubini-Study symplectic form 
	in an affine chart, one can compute that the moment map $H_i:\mathbb X\to [0,1]$ associated to the $S^1$-action $\theta_i$ is the following
	\[
	H_i= \frac{\sum_{k\neq i}|Z_{ik}|^2}{\sum_{j<k} |Z_{jk}|^2} 
	\]
	Here we use the convention that $dH=-\iota_X \omega$.

	Specifically, further clarifying the frozen variables $g_{i,i+1}$, we have
	\begin{align*}
		H_1&=(|g_{12}|^2+|Z_{13}|^2+|g_{14}|^2)/\textstyle\sum_{j<k}|Z_{jk}|^2 \\
		H_{2}&=(|g_{12}|^2+|g_{23}|^2+|Z_{24}|^2)/\textstyle\sum_{j<k}|Z_{jk}|^2 \\
		H_{3}&=(|Z_{13}|^2+|g_{23}|^2+|g_{34}|^2)/\textstyle\sum_{j<k}|Z_{jk}|^2 \\
		H_{4}&=(|g_{14}|^2+|Z_{24}|^2+|g_{34}|^2)/\textstyle\sum_{j<k}|Z_{jk}|^2
	\end{align*}
	Since $H_1+H_2+H_3+H_4=2$, we obtain a well-defined map
	\[
	\mu =(H_1,H_2,H_3,H_4) : \mathbb X \to \Delta_{2,4}
	\]
	where $\Delta_{2,4}$ denotes the $(2,4)$-hypersimplex. Recall that the $(k,n)$-hypersimplex is defined by
	\[
	\Delta_{k,n}=\{  x=(x_1,\dots, x_n)\in [0,1]^n \mid x_1+\cdots+x_n=k\}
	\]
	It is known that $\Delta_{2,4}$ is an octahedron as in Figure \ref{fig:wall}.
	
	\begin{figure}
		\centering
		\includegraphics[width=0.5\linewidth]{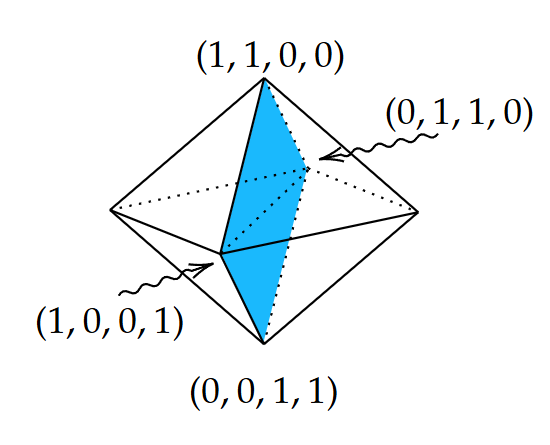}
		\caption{
			\footnotesize The octahedron. The blue region depicts a hyperplane section, which will be the "wall". Its four vertices are presented with coordinates $(x_1,x_2,x_3,x_4)$ of $\Delta_{2,4}$.}
		\label{fig:wall}
	\end{figure}
	
	There are exactly 6 vertices given by setting two of $x_i$'s to be 0 and the other two of $x_i$'s to be 1.
	Each facet of $\Delta_{2,4}$ is given by the convex hull of three distinct vertices.
	Each of these functions ranges over $[0,1]$. It takes value 1 if only two $g$'s and one $Z$ is nonvanishing, while it takes value 0 similarly. In particular it cannot take boundary values on $X=\mathbb X\setminus D_{ac}$. 
	Thus, 
	\[
	\mu : X\to \Delta^{\circ}_{2,4}
	\]
	where $\Delta_{2,4}^\circ$ is the interior of $\Delta_{2,4}$.
	

	There are two cluster charts
	\begin{align*}
		U_{13} &= \{z \in X \mid  Z_{13}(z)\neq 0\} = X \setminus D_{13}\cong (\mathbb C^*)^4   \\  
		U_{24}&=\{z\in X \mid Z_{24}(z)\neq 0\} = X\setminus D_{24}\cong (\mathbb C^*)^4
	\end{align*}
	
	The complement of the union $U_{13}\cup U_{24}$ is a codimension-2 subvariety $D_{13}\cap D_{24}=\{Z_{13}=Z_{24}=0\}$ whose image under $  \mu$ is given by the subset 
	\begin{equation} 
		\label{eq_Pi}
		\Pi:= \{ x_1+x_3=x_2+x_4=1\}
	\end{equation}
	within $\Delta_{2,4}$. Similarly, the image of $D_{12}\cap D_{34}$ is $\{x_1+x_2=x_3+x_4=1\}$, and the image of $D_{23}\cap D_{14}$ is $\{x_1+x_4=x_2+x_3=1\}$.
	
	\subsection{Lagrangian fibration}
	\label{s_lag_fib}
	Define 
	\[B= \Delta_{2,4}^\circ \times \mathbb R_\rho
	\]
	Consider the map
	\[
	\pi= \left( \mu, \log\left| \frac{g_{23}g_{14}}{g_{12}g_{34}}\right| \right): X=\mathbb X\setminus D_{ac}\to B 
	\]
	Define
	\begin{equation} 
		\label{eq_Gamma_singular_locus}
		\Gamma =\Pi\times\{0\} := \{ (x,\r) \in B \mid \r=0, \ x_1+x_3=x_2+x_4=1 \}
	\end{equation}
	and
	define
	\[B_0:=B\setminus \Gamma
	\]
	Denote by $L_q$ or $L_{x, \rho}$ the $\pi$-fiber at $q=(x, \rho)\in B$.

	\begin{prop}
		\label{singular_locus_prop}
		The above map $\pi$ is a Lagrangian fibration, and its singular locus is $\Gamma$.
	\end{prop}
	\begin{proof}
		Consider first the auxiliary conic-like fibration $$f: X\rightarrow \C,  \qquad p:=[Z_{ij}]\mapsto \frac{Z_{13}Z_{24}}{g_{12} g_{34}}=:f(p)$$
		Now $e^\rho$ can be interpreted as the distance from $f(p)$ to 1. Away from the anti-canonical divisor, $e^\rho$ is never zero or infinity, so $f$ actually takes values in $\C\backslash\{1\}$. The fibers of $f$ are generically smooth complex varieties of dimension three.
		Moreover, the critical points lie in $\{Z_{13}=Z_{24}=0\}$, and the only singular fiber is $f^{-1}(0)$. 
		
		The level set $\mu^{-1}(x)$ surjects onto $\C$ through $f$. The Hamiltonian actions associated to $H_i$ preserve the $f$-fiber. One may notice that the $\pi$-fiber $L:=L_{  x, \rho}$ is precisely the subset of the level set $\mu^{-1}(x)$ whose $f$-image is the concentric circles centered at 1 with radius $e^\r$. Assume $L$ does not meet $\{Z_{13}=Z_{24}=0\}$, it is clear that $L$ is $\T^4$.
		We aim to show that it is a Lagrangian submanifold. Fix $p\in L$.
		Each Hamiltonian action $\theta_i$ above preserves the $f$-fiber and hence sweeps out a 3-torus $\T^3$ within a $\pi$-fiber. The Hamiltonian vectors give a 3-dimensional subspace $S$ inside $T_p L$.
		Let $V_c\in T_p L$ be a vector complementary to $S$.
		It suffices to show for any $V\in S$, we have $\omega(V_c,V)=0$.
		Indeed, let's consider the symplectic reduction.
		Denote by $q: \mu^{-1}(x)\to \mu^{-1}(x)/\T^3$ the quotient map. 
		Denote by $\omega_{red}$ the reduced symplectic form. Then, as $dq(S)=0$, we conclude
		\[
		\omega(V_c,V) = \omega_{red}(dq(V_c), dq(V))=0
		\]
		When $L$ does meet $\{Z_{13}=Z_{24}=0\}$, or equivalently when $(x,\r)\in \Gamma$, the same argument carries over to show that it is a (singular) Lagrangian submanifold. 
	\end{proof} 
	

	Define
	\begin{align*}
		&\Theta_{13}:=\bar\Theta_{13}\setminus \Gamma =\{ (x,0) \mid x_1+x_3 < x_2+x_4 \} \\
		&\Theta_{24}:=\bar\Theta_{24}\setminus \Gamma = \{ (x,0) \mid x_1+x_3 > x_2+x_4 \}
	\end{align*}
	Then, observe that the $\pi$-images of $D_{13}$ and $D_{24}$ are their closures
	$\bar\Theta_{13}:=\pi(D_{13})=\{ (x,0) \mid x_1+x_3\leqslant x_2+x_4 \}$ and $\bar\Theta_{24}:= \pi(D_{24})= \{ (x,0) \mid x_1+x_3\geqslant x_2+x_4 \}$.
	The following statement says that $\Theta_{13}$ and $\Theta_{24}$ are the walls of Maslov-0 holomorphic disks.

	\begin{prop}
		A smooth Lagrangian fiber $L_{x,\rho}$ for $(x,\rho)\in B_0$ bounds a nontrivial Maslov index zero holomorphic disk if and only if $(x,\rho)\in \Theta_{13}\cup \Theta_{24}$.
	\end{prop}
	
	\begin{proof}
		In this set-up, there is a relation \cite[Lemma 3.1]{AuTDual} that connects the Maslov index of a holomorphic disk representing a class $\b\in\pi_2(X,L)$ and its intersection with the anti-canonical divisor, i.e.  $\mu(\beta)=2\b\cdot D_{ac}$. Being of Maslov index zero means in particular that the disk does not intersect the anti-canonical divisor. Maximum principle then implies that the $f$-image of the disk is constant. However the smooth $f$-fibers are all diffeomorphic to $(\C^\ast)^3$ and do not contain a disk for purely topological reasons. Hence a holomorphic disk of Maslov index zero, if exists, must be bounded by $f^{-1}(0)$. The action coordinate of a smooth Lagrangian fiber which intersects non-trivially with $f^{-1}(0)$ is always in $\Theta_{13}\cup \Theta_{24}$. This proves the "only if" part.
		
		To prove the "if" part, notice that a smooth Lagrangian fiber over $\Theta_{13}\cup \Theta_{24}$ must intersect non-trivially with $f^{-1}(0)$. Hence it always bounds a Maslov zero disk: A Lagrangian fiber over $(x,\r)\in \Theta_{13}$ contains a point that looks like
		\[[\tilde Z_{13}:0:\tilde g_{12}:\tilde g_{23}:\tilde g_{34}:\tilde g_{41}] \]with all coordinates with a tilde nonvanishing. This point also belongs to $f^{-1}(0)$. Then the following holomorphic disk is bounded by exactly the same Lagrangian fiber and is of Maslov index zero
		\[\phi:\mathbb D\rightarrow Gr(2,4),z\mapsto [zZ_{13}:0:g_{12}:g_{23}:g_{34}:g_{41}]\]
		The same proof works for $(x,\r)\in \Theta_{24}$. 
	\end{proof}

	\subsection{Local systems of topological disks}
	\label{s_local_system_topo}
	Recall that $\mathbb X=Gr(2,4)$. We are interested in the following two local systems
	\begin{align*}
		\mathscr R := \bigcup_{q\in B_0} \pi_2(\mathbb X, L_q)   \ , \qquad \ \mathscr S := \bigcup_{q\in B_0} \pi_1(L_q) 
	\end{align*}
	
	Denote the class of the complex line in $\mathbb X$ by $\mathcal H$. Recall that $\X=Gr(2,4)_\C$. 
	Since the exact sequence 
	\[
	\xymatrix{
		0=\pi_2(L_q)  \ar[r] &  \pi_2(\mathbb X) \ar[r] & \pi_2(\mathbb X, L_q) \ar[r]^{\partial\quad} & \pi_1(L_q)\cong \mathbb Z^4 \ar[r] &  0=\pi_1(X)
	}
	\]
	splits and (the image of) $\pi_2(\mathbb X)$ is generated by $\cal H$, we see that $\pi_2(\mathbb X,L_q)$ is in fact a free abelian group generated by four disks and $\cal H$.
	On the other hand, as $q$ travels along a loop around the singular locus, it may feature a non-identity automorphism $\mathscr R|_q\to \mathscr R|_q$.
	Moreover, as the restriction maps for the local systems $\mathscr R$ and $\mathscr S$ commute under the natural boundary maps, one concludes that the monodromy for $\mathscr S$ is identical to the one for $\mathscr R$.
	However, the sections of $\mathscr R$ are more convenient to specify.

	To describe such monodromy, we proceed the following three steps:
	\begin{itemize}
		\itemsep 2pt
		\item Cover $B_0$ by contractible open subsets.
		\item Specify generating sections over these open subsets. 
		\item Study how these sections are related over the overlaps.
	\end{itemize}
	\vspace{0.5em}

	First, we put $U_+^{\prime}=\{\rho>0\}\subseteq B_0$ and $U_-^{\prime}=\{\rho<0\}\subseteq B_0$.
	Let $\mathcal N_{13}$ and $\mathcal N_{24}$ be sufficiently small neighborhoods of $\Theta_{13}$ and $\Theta_{24}$ inside $B_0$ respectively.
	Then, we define 
	\begin{align*}
		U_+&=U_+^{\prime}\cup \mathcal N_{13}\cup \mathcal N_{24}\\
		U_-&=U_-^{\prime}\cup \mathcal N_{13}\cup \mathcal N_{24}
	\end{align*}
	i.e. slight thickenings of $U'_+$ and $U'_-$ respectively.
	Note that $\{ U_+, U_-\}$ gives a covering of $B_0$. Since $U_\pm$ are contractible, the local systems can be trivialized over $U_\pm$.
	On the other hand, we define
	\begin{align*}
		U_{13}=B_0\setminus \mathcal N_{13} \\
		U_{24}=B_0\setminus \mathcal N_{24}
	\end{align*}
	and $\{U_{13}, U_{24}\}$ also forms an open covering of $B_0$.

	To be specific, let's pick $x_0=(\frac{1}{2},\frac{1}{2},\frac{1}{2},\frac{1}{2})\in \Delta_{2,4}^\circ$, and pick two points $q_+=(x_0,  2\log R) \in U_+$, $q_-=(x_0,-2\log R)\in U_-$.
	Topologically,
	the sections in 
	\[
	\mathscr R(U_\pm')\cong  \pi_2(\mathbb X,L_{q_\pm})
	\]
	can be specified by the intersections with the six coordinate hyperplanes. In fact, the intersection numbers define a linear homomorphism
	\begin{align*}
		\mathcal I_\pm: 
		\mathscr R(U_\pm') \to \mathbb Z^6 \ , \qquad \beta\mapsto (\beta\cdot D_{13}, \beta\cdot D_{24}, \beta\cdot D_{12},\beta\cdot D_{23},\beta\cdot D_{34},\beta\cdot D_{14})
	\end{align*}
	The line class $\mathcal H$ can be viewed as a global section of $\mathscr R$ over $B_0$ such that each intersection number is $1$. Namely, $\mathcal I_\pm(\mathcal H)=(1,1,1,1,1,1)$.
	By forgetting $\beta\cdot D_{24}$ and $\beta\cdot D_{13}$ respectively, we also need the maps
	\begin{align*}
		\mathcal I_{13}: 
		\mathscr R(U_{13}) \to \mathbb Z^5 \ , \qquad \beta\mapsto (\beta\cdot D_{13}, \beta\cdot D_{12},\beta\cdot D_{23},\beta\cdot D_{34},\beta\cdot D_{14})  \\
		\mathcal I_{24}: 
		\mathscr R(U_{24}) \to \mathbb Z^5 \ , \qquad \beta\mapsto (\beta\cdot D_{24}, \beta\cdot D_{12},\beta\cdot D_{23},\beta\cdot D_{34},\beta\cdot D_{14})
	\end{align*}
	
	\subsubsection{Explicit disks}
	\label{s_explicit_disks}
	The following constructions will be useful.
	Let $z$ represent a complex variable in the unit disk $\mathbb D\subseteq \mathbb C$.
	Define the six holomorphic disks $v_{12}^+,v_{34}^+,u^+_{23,13}, u^+_{23,24}, u^+_{14,13}, u^+_{14,24}: (\mathbb D,\partial\mathbb D)\to (\mathbb X, L_{q_+})$ by 
	\begin{table}[H]
		\renewcommand{\arraystretch}{1.5} 
		\centering
		\begin{tabular}{c|c|c|c|c|c|c}
			& $Z_{13}$  &  $Z_{24}$ & $g_{12}$ & $g_{23}$& $g_{34}$ & $g_{14}$ \\\hline
			$v_{12}^+$     & $\sqrt{R^2+z}$ & $\sqrt{R^2+z}$ & $z$ & $R$ & $1$ & $R$ \\\hline 
			$v_{34}^+$     & $\sqrt{R^2+z}$ & $\sqrt{R^2+z}$ & 1 & $R$ & $z$ & $R$ \\\hline
			$u_{23,13}^+$ & $\frac{1+R^2z}{R^2+z}\sqrt{R^2+z}$ & $\sqrt{R^2+z}$ & 1 & $Rz$ & 1 & $R$ \\\hline
			$u_{23,24}^+$ & $\sqrt{R^2+z}$ & $\frac{1+R^2z}{R^2+z}\sqrt{R^2+z}$ & 1 & $Rz$ & 1 & $R$ \\\hline
			$u_{14,13}^+$ & $\frac{1+R^2z}{R^2+z}\sqrt{R^2+z}$ & $\sqrt{R^2+z}$ & 1 & $R$ & 1 & $Rz$ \\\hline
			$u_{14,24}^+$ & $\sqrt{R^2+z}$ & $\frac{1+R^2z}{R^2+z}\sqrt{R^2+z}$ & 1 & $R$ & 1 & $Rz$ \\\hline
		\end{tabular}
	\end{table}
	
	One can check that their boundaries are indeed contained in $L_{q_+}$.
	Then, we define
	\begin{equation}
		\label{alpha_beta_+_eq}
		\alpha_{12}^+,\alpha_{34}^+,\beta^+_{23,13}, \beta^+_{23,24}, \beta^+_{14,13}, \beta^+_{14,24} \in \mathscr R(U_+')
	\end{equation}
	to be the sections of the sheaf $\mathscr R$ over $U_+$ such that their stalks at $q_+$ are represented by the above disks $v_{12}^+,v_{34}^+,u^+_{23,13}, u^+_{23,24}, u^+_{14,13}, u^+_{14,24}$ respectively. 
	Under $\mathcal I_\pm$, the images of the above disk classes can be described by the following table of intersection numbers:
	\begin{table}[H]
		\centering
		\begin{tabular}{c|c|c|c|c|c|c}
			& $D_{13}$  &  $D_{24}$ & $D_{12}$ & $D_{23}$& $D_{34}$ & $D_{14}$ \\\hline
			$\alpha_{12}^+$     &   &   & $1$ &   &   &   \\\hline 
			$\alpha_{34}^+$     &   &   &   &   & $1$ &   \\\hline
			$\beta_{23,13}^+$ & $1$ &   &   & $1$ &   &   \\\hline
			$\beta_{23,24}^+$ &   & $1$ &   & $1$ &   &   \\\hline
			$\beta_{14,13}^+$ & $1$ &   &   &   &   & $1$ \\\hline
			$\beta_{14,24}^+$ &   & $1$ &   &   &   & $1$ \\\hline
		\end{tabular}
	\end{table}
	
	A schematic figure is depicted below in Figure \ref{fig:disks}. We use different colors to label the six holomorphic disks.
	\begin{center}
		\begin{figure}[H]
			\centering
			\includegraphics[width=0.8\linewidth]{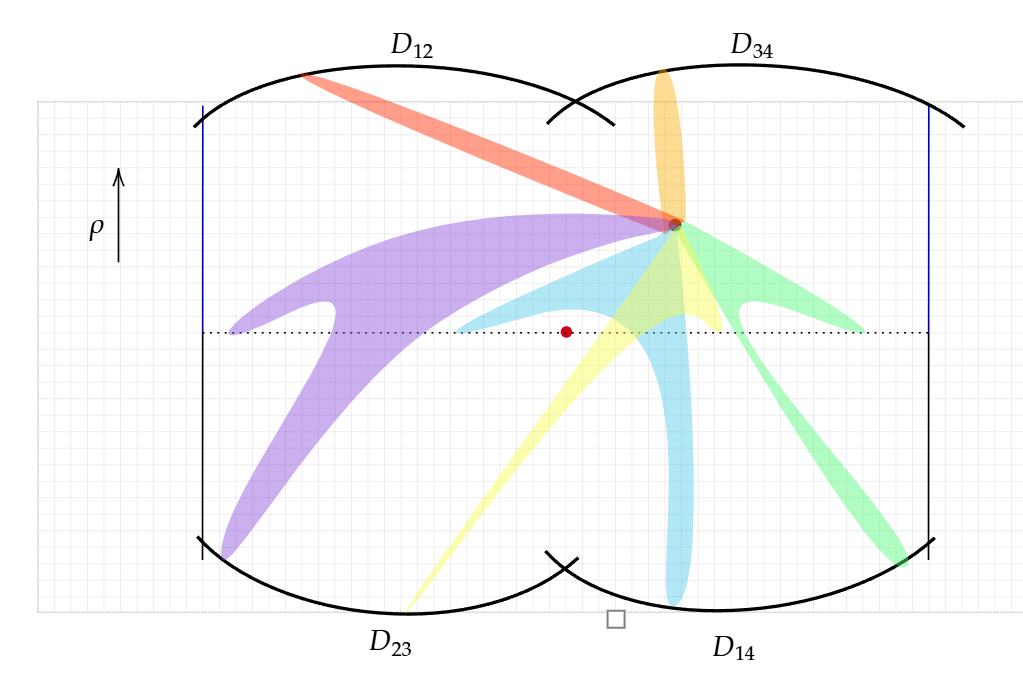}
			\caption{The six holomorphic disks bounded by a regular fiber over $U_+$}
			\label{fig:disks}
		\end{figure}
	\end{center}
	
	Similarly, one can find (holomorphic) disks $v_{23}^-,v_{14}^-,u_{12,13}^-,u_{12,24}^-, u_{34,13}^-, u_{34,24}^-:(\mathbb D,\partial\mathbb D)\to (\mathbb X,L_{q_-})$ as below:
	\begin{table}[H]
		\centering
		\renewcommand{\arraystretch}{1.5} 
		\begin{tabular}{c|c|c|c|c|c|c}
			& $Z_{13}$  &  $Z_{24}$ & $g_{12}$ & $g_{23}$& $g_{34}$ & $g_{14}$\\\hline
			$v_{23}^-$& $\sqrt{R^2+z}$& $\sqrt{R^2+z}$& $R$ & $z$& $R$& $1$\\ \hline 
			$v_{14}^-$& $\sqrt{R^2+z}$& $\sqrt{R^2+z}$& $R$& $1$& $R$ & $z$\\\hline
			$u_{12,13}^-$& $\frac{1+R^2z}{R^2+z}\sqrt{R^2+z}$& $\sqrt{R^2+z}$& $Rz$& $1$& $R$ & $1$\\\hline
			$u_{12,24}^-$& $\sqrt{R^2+z}$& $\frac{1+R^2z}{R^2+z}\sqrt{R^2+z}$& $Rz$& $1$& $R$& $1$\\\hline
			$u_{34,13}^-$& $\frac{1+R^2z}{R^2+z}\sqrt{R^2+z}$& $\sqrt{R^2+z}$& $R$& $1$& $Rz$& $1$\\\hline
			$u_{34,24}^-$& $\sqrt{R^2+z}$& $\frac{1+R^2z}{R^2+z}\sqrt{R^2+z}$& $R$& $1$& $Rz$& $1$\\\hline
		\end{tabular}
	\end{table}
	By requiring that the stalks at $q_-$ are represented by them respectively, we also obtain the sections
	\begin{equation}
		\label{alpha_beta_-_eq}
		\alpha_{23}^-,\alpha_{14}^-, \beta_{12,13}^-,\beta_{12,24}^-, \beta_{34,13}^-, \beta_{34,24}^- \in \mathscr R(U_-')
	\end{equation}
	They admit intersections
	
	\begin{table}[H]
		\centering
		\begin{tabular}{c|c|c|c|c|c|c}
			& $D_{13}$  &  $D_{24}$ & $D_{12}$ & $D_{23}$& $D_{34}$ & $D_{14}$\\\hline
			$\alpha_{23}^-$&  &  &   & $1$&  &  \\ \hline 
			$\alpha_{14}^-$&  &  &  &  &   & $1$\\\hline
			$\beta_{12,13}^-$& $1$&  & $1$&  &   &  \\\hline
			$\beta_{12,24}^-$&  & $1$& $1$&  &  &  \\\hline
			$\beta_{34,13}^-$& $1$&  &  &  & $1$&  \\\hline
			$\beta_{34,24}^-$&  & $1$&  &  & $1$&  \\\hline
		\end{tabular}
	\end{table}

	\begin{lem}
		$\mathcal I_{13}$ and $\mathcal I_{24}$ are isomorphisms. \label{lemIiso}
	\end{lem}
	
	\begin{proof}
		Without loss of generality, we only consider the case of $\mathcal I_{13}$.
		Suppose $\beta\in \mathscr R(U_{13})$ lies in the kernel of $\mathcal I_{13}$,
		and let $u$ be a representative of $\beta$. Due to the vanishing of intersection numbers, we may assume $u$ lies in $\mathbb X\setminus (D_{13}\cup D_{12}\cup D_{23}\cup D_{34}\cup D_{14})\cong (\mathbb C^*)^4$. However, $\pi_2( (\mathbb C^*)^4, L_{q_+})$ is trivial.
		This shows the injectivity.
		
		For the surjectivity, we first note that $\mathscr R(U_+')\cong \mathscr R(U_{13})$ as both $U_+'$ and $ U_{13}$ are contractible.
		By using the explicit disks, one can see the standard basis of $\mathbb Z^5$ lies in the $\mathcal I_{13}$-image.
		In fact, observe that 
		\begin{align*}
			\mathcal I_{13} (\mathcal H - \alpha_{12}^+ - \beta_{23,24}^+ - \alpha_{34}^+ - \beta_{14,24}^+) &= (1,0,0,0,0) \\
			\mathcal I_{13}(\alpha_{12}^+)  &=  (0,1,0,0,0)   \\
			\mathcal I_{13}(\beta_{23,24}^+)&=(0,0,1,0,0) \\
			\mathcal I_{13}(\alpha_{34}^+)&=(0,0,0,1,0) \\
			\mathcal I_{13}(\beta_{14,24}^+)&=(0,0,0,0,1)
		\end{align*}
		This completes the proof.
	\end{proof}
	
	These explicitly constructed disks are important later in the discussions of superpotentials; here, however, we only use their topological properties.

	\subsubsection{Detection of global sections}
	For clarity, we would like to find as many global sections of $\mathscr R$ over $B_0$, supplemented by a minimal amount of locally defined sections. 
	The above discussion enables one to identify the sections of disk classes through these intersection numbers.
	
	Let $p_{13}: \mathbb Z^6\to \mathbb Z^5$ and $p_{24}:\mathbb Z^6\to\mathbb Z^5$ be the projection maps forgetting the first and second components respectively.
	Then, by definition, we have the following commutative diagrams
	\[
	\xymatrix{
		\mathscr R(U_{13}) \ar[d]^{\mathcal I_{13}} \ar[rr]^{r_{13,\pm}}_{\cong} & & \mathscr R(U_\pm') \ar[d]^{\mathcal I_\pm} & & \mathscr R(U_{24}) \ar[ll]_{r_{24,\pm}}^{\cong}  \ar[d]^{\mathcal I_{24}} \\
		\mathbb Z^5 & & \mathbb Z^6 \ar[ll]_{p_{13}} \ar[rr]^{p_{24}} & & \mathbb Z^5
	}
	\]
	Here $r_{13,\pm}$ and $r_{24,\pm}$ denote the restriction maps of the sheaf $\mathscr R$. They are actually isomorphisms since $U_{13}, U_{24}, U_\pm'$ are contractible.

	Regarding the covering $\{U_{13},U_{24}\}$ of $B_0$, two sections $\gamma'\in \mathscr R(U_{13})$ and $\gamma''\in\mathscr R(U_{24})$ glue to a global section of $\mathscr R$ over $B_0$ if and only if $r_{13,\pm}(\gamma') = r_{24,\pm}(\gamma'')$.
	Meanwhile, regarding the covering $\{U_+,U_-\}$ and the natural identifications $\mathscr R(U_\pm)\cong \mathscr R(U_\pm')$, the use of the intersection numbers yields the following:
	
	\begin{lem}[A criterion to detect global sections]
		\label{criterion_global_lem}
		Let $\gamma_+ \in\mathscr R(U_+')$ and $\gamma_- \in\mathscr R(U_-')$.
		There exists a global section $\gamma\in \mathscr R(B_0)$ such that $\gamma$ restricts to $\gamma_+$ and $\gamma_-$ if and only if
		\[
		\mathcal I_+(\gamma_+) =\mathcal I_-(\gamma_-) \quad \in \mathbb Z^6
		\]
	\end{lem}
	
	\begin{proof}
		By the above commutative diagrams,
		\[
		\mathcal I_{13}\circ (r_{13,+})^{-1}( \gamma_+)
		=
		p_{13} \circ \mathcal I_+( \gamma_+) = p_{13}\circ \mathcal I_-(\gamma_-)
		=
		\mathcal I_{13}\circ (r_{13,-})^{-1}( \gamma_-)
		\]
		Since $\mathcal I_{13}$ is an isomorphism, we conclude that $(r_{13,+})^{-1}( \gamma_+)=(r_{13,-})^{-1}( \gamma_-)=:\gamma'$.
		Similarly, we also have $(r_{24,+})^{-1}( \gamma_+)=(r_{24,-})^{-1}( \gamma_-)=:\gamma''$.
		Then, it is easy to check $\gamma'\in\mathscr R(U_{13})$ and $\gamma''\in\mathscr R(U_{24})$ glue to a global section $\gamma$.
		The reverse is also true.
	\end{proof}
	
	\begin{ex}
		There is an obvious case for the above lemma: $\mathcal I_+(\mathcal H|_{U_+'})=\mathcal I_-(\mathcal H|_{U_-'})=(1,1,1,1,1,1)$.
	\end{ex}
	
	\begin{prop}
		\label{prop_gamma_i}
		For $i=1,2,3,4$, there is a global section $\gamma_i$ of the local system $\mathscr R$ such that $\gamma_i\cdot D_{ij}=1$ for $j\neq i$ and $\gamma_i\cdot D_{jk}=0$ for $j\neq i\neq k$.
		Moreover, one has $\gamma_1+\gamma_2+\gamma_3+\gamma_4=2\mathcal H$.
	\end{prop}
	
	\begin{proof}
		Without loss of generality, we may fix $i=1$.
		The proof is by construction.
		Consider $\gamma_+:=\alpha_{12}^++\beta^+_{14,13}$ and $\gamma_-:=\alpha_{14}^-+\beta^-_{12,13}$.
		Then, $\mathcal I_+(\gamma_+)=(1,0,1,0,0,1)$ and $\mathcal I_-(\gamma_-)=(1,0,1,0,0,1)$.
		Thus, the result follows from Lemma \ref{criterion_global_lem}.
	\end{proof}

	\begin{prop}
		\label{prop_gamma_i_partial}
		For $i=1,2,3,4$, $\partial \gamma_i$ is the global section of $\mathscr S$ that represents the orbit of the Hamiltonian $S^1$-action $\theta_i$ in (\ref{Ham_S_1_theta_1234}).
	\end{prop}
	
	\begin{proof}
		Without loss of generality, let's fix $i=1$.
		Since the orbit may also be viewed as a global section of $\mathscr S$, it suffices to verify the agreement at a specific point $q$ in $B_0$.
		We choose $q=(( \frac{2}{3}, \frac{1}{3}, \frac{2}{3}, \frac{1}{3}), 0)\in \Theta_{24}$, as the stalk $\gamma_1(q)\in \pi_2(\mathbb X, L_q)$ can be represented explicitly by the map setting $[Z_{13}:Z_{24}:g_{12}:g_{23}:g_{34}:g_{14}]=[2z: 1 : z: 1: 1: z]$, where $z\in\mathbb D$.
		By definition, the boundary $\partial\gamma_i(q)$ is the desired $S^1$-orbit.
	\end{proof}
	
	\subsubsection{Preferred sections and monodromy}\label{subsec_secandmono}
	We would like to pick an integral basis of $\mathscr R(U_\pm)$ that consists of as many global sections as possible. The number of independent global sections is four. 
	As the rank of $\mathscr R$ is five, one has to supplement the integral basis by another \textit{local} section on $U_+$ and $U_-$, respectively. Certain ad hoc choices seem to be unavoidable.
	We make the following choice of sections
	\begin{equation}
		\label{integral_basis_eq}
		\begin{aligned}
			\mathscr R(U_+):&=\Z\langle \gamma_1,\gamma_2,\gamma_3, \mathcal H,\a_{34}^+\rangle \\ 
			\mathscr R(U_-):&=\Z\langle\g_1,\g_2,\g_3,\mathcal H,\b_{34,24}^-\rangle
		\end{aligned}   
	\end{equation}
	The fact that the sections are linearly independent and form an integral basis can be checked by straightforward computations, using the isomorphisms $\mathcal I_{13}$ and $\mathcal I_{24}$. Indeed, the standard unit vectors $e_i=(0,...,1,0,...)$ for each $i\in \{1,...,5\}$ can be obtained as follows
	\begin{align}
		\text{ On } U_+: \begin{cases}
			e_{1}&=\mathcal I_{13}(\H-\g_1-\g_3)\\
			e_{2}&=\mathcal I_{13}(\g_1+\g_2+\a_{34}^+-\H)\\
			e_{3}&=\mathcal I_{13}(\g_3-\a_{34}^+)\\
			e_{4}&=\mathcal I_{13}(\a_{34}^+)\\
			e_{5}&=\mathcal I_{13}(\H-\g_2-\a_{34}^+)\\
		\end{cases}
		\hspace{3em}
		\text{ On } U_-: \begin{cases}
			e_{1}&=\mathcal I_{13}(\H-\g_1-\g_3)\\
			e_{2}&=\mathcal I_{13}(2\g_1+\g_2+\g_3+\b_{34,24}^--2\H)\\
			e_{3}&=\mathcal I_{13}(\H-\g_1-\b_{34,24}^-)\\
			e_{4}&=\mathcal I_{13}(\g_1+\g_3+\b_{34,24}^--\H)\\
			e_{5}&=\mathcal I_{13}(2\H-\g_1-\g_2-\g_3-\b_{34,24}^-)\\
		\end{cases}
	\end{align}
	Notice that one cannot take $\g_4$ in place of $\H$. For instance, the image under $\mathcal I_{13}$ of $\Z\langle \g_1,\g_2,\g_3,\g_4,\a_{34}^+\rangle$ does not contain $e_1$ or $e_2$; only $2e_1$ and $2e_2$ can be realized.
	
	\begin{rmk}
		More precisely, the above disk classes are defined initially either in $U_+^\prime$ and $U_-^\prime$. Take $\a_{34}^+$ for instance. It is certainly true that its extension to the entire $U_+$ is not unique, but its extensions to $U_+^\prime\cup \mathcal N_{24}$ is unique, by the dictated isomorphisms. Similarly, the extension to $U_+^\prime\cup \mathcal N_{13}$ is unique. Gluing these two pieces produces a unique section on $U_+$, which is the desired unique extension of $\a_{34}^+$. In this sense, the extension of $\a_{34}^+$ to $U_+$ is unique. The same process works for $\b_{34,24}^-$. Slightly abusing notation, the extension to $U_+$ or $U_-$ is still denoted by $\a_{34}^+$ or $\b_{34,24}^-$. In practice, this simply means forgetting the intersection with both $D_{13}$ and $D_{24}$. 
	\end{rmk}

	The disks constructed in Section \ref{s_explicit_disks} can be expressed in terms of the basis chosen above as follows
	\begin{equation}
		\label{eq_alpha_to_gamma_H}    
		\text{ In }\mathscr R(U_+):
		\begin{cases}
			\alpha_{12}^+ 
			&  = \gamma_1 +\gamma_2+\a_{34}^+ -\mathcal H \\
			\beta_{23,13}^+  
			&= \g_3-\a_{34}^+\\
			\beta_{23,24}^+
			&=\mathcal H-\g_1-\a_{34}^+ \\
			\beta_{14,13}^+
			&= \mathcal H-\g_2-\a_{34}^+\\
			\beta_{14,24}^+
			&=\g_4-\a_{34}^+ 
		\end{cases}
		\hspace{3em}
		\text{ In }\mathscr R(U_-):\begin{cases}
			\alpha_{23}^- 
			&=\mathcal H-\g_1-\b_{34,24}^- \\
			\a_{14}^-&=\g_4-\b_{34,24}^-\\
			\beta_{12,13}^-&=\g_1-\g_4+\b_{34,24}^-\\
			\b_{12,24}^-&= \gamma_1 +\gamma_2+\b_{34,24}^- -\mathcal H \\
			\beta_{34,13}^-&=\mathcal H-\g_2+\b_{34,24}^--\g_4
		\end{cases}
	\end{equation}
	\label{disks}

	This enables one to compute the monodromy of our integral basis. Global sections are not subject to the monodromy; the only source for monodromy is the non-globality of the local sections we choose. We extend $\a_{34}^+$ first through $\mathcal N_{24}$
	then through $\mathcal N_{13}$. On $\mathcal N_{24}$ one has 
	\begin{align}
		\text{ On }\mathcal N_{24}: \a_{34}^+=\b_{34,24}^-,\hspace{3em} \text{ On }\mathcal N_{13}:\b_{34,24}^-= \mathcal H-\g_1-\g_3+\a_{34}^+\label{eq_disk_monodromy}
	\end{align}
	
	
	\subsection{Integral affine atlas} \label{s_int_aff_at}
	The integral affine structure on the base is given by the action coordinates of Lagrangian fibers.
	Remark that the action coordinates can be viewed as the symplectic area of the cylinder with boundary on the adjacent Lagrangian torus fibers. See Figure \ref{fig:cylinder} and \cite{action_angle}.
	\begin{figure}[h]
		\centering
		\includegraphics[width=0.5\linewidth]{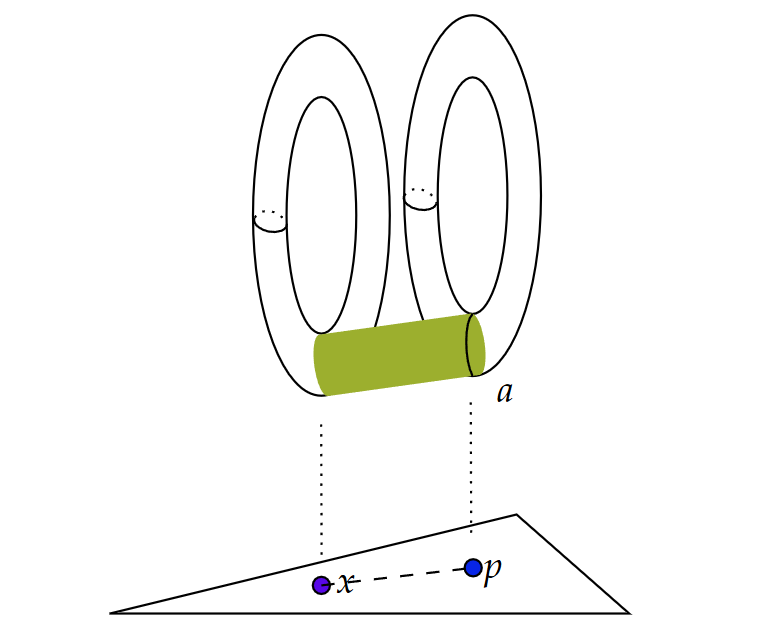}
		\caption{The geometric interpretation of the action coordinates.}
		\label{fig:cylinder}
	\end{figure}
	
	Inspired by this, we are interested in the symplectic areas of topological disks in $\pi_2(X,L_q)$ and view them as functions of $q$. By considering local trivializations of the local systems $\mathscr R =\bigcup_q \pi_2(\mathbb X, L_q)$ and $\mathscr S=\bigcup_q \pi_1(L_q)$, we may potentially find concrete action coordinates on the base.
	
	Recall that there are natural global sections $\gamma_i$ of $\mathscr R$ as in Proposition \ref{prop_gamma_i}.
	Then, we can view $E(\gamma_i)$ as a real-valued function on $B$.
	
	\begin{prop}
		\label{prop_E(gamma_i)}
		$E(\gamma_i)\circ \pi=H_i$ as functions on $X$ for $i=1,2,3,4$. Moreover, the symplectic area of the complex line is $E(\mathcal H)=1$.
	\end{prop}
	\begin{proof}
		Given a fixed point $q$ in the base, we represent $\gamma_1(q)\in \pi_2(\mathbb X, L_q)$ by a map $u: (\mathbb D, \partial \mathbb D) \to (\mathbb X, L_q)$. Since $\gamma_1\cdot D_{34}=0$, we may require that $u$ does not meet $D_{34}=\{Z_{34}=0\}$ and can be regarded as a map into an affine chart $U\cong \mathbb C^5$ with coordinates $z_{ij}=Z_{ij}/Z_{34}$.
		By Proposition \ref{prop_gamma_i_partial}, we may assume that the tangent vector of $\partial u : S^1 = \mathbb R/\mathbb Z\to U$ at $t$ is given by
		$V(t)=2\pi \mathbf i \ \left(  z_{12} \partial_{z_{12}} -  \bar z_{12} \partial_{\bar z_{12}} + z_{13} \partial_{z_{13}} -  \bar z_{13} \partial_{\bar z_{13}} +z_{14} \partial_{z_{14}} -  \bar z_{14} \partial_{\bar z_{14}} \right)$.
		On the above affine chart $U$, we have $\omega=d\lambda$ where 
		\[
		\lambda=\frac{\mathbf i}{2\pi} \bar \partial \log (1+ |z_{12}|^2+|z_{13}|^2+ |z_{14}|^2+ |z_{23}|^2+|z_{24}|^2)
		\]
		Therefore, by Stokes' formula and by direct computations,
		\begin{align*}
			E(\gamma_1)(q)= \int_{\mathbb D} u^*\omega = \int_0^1 \lambda (V(t)) = H_1
		\end{align*}
		Similarly, we can compute $E(\gamma_i)$ for $i=2,3,4$.
		On the other hand, 
		one can represent the line $\cH$ by a holomorphic sphere $\C\P^1$, parametrized by $[s:t]\mapsto[0:0:s:0:t:0]$, which is contained in $Gr(2,4)_\C$.
		One can compute directly that the symplectic area of this sphere is 1. Indeed, work in the chart $s\neq 0$, and the sphere is parametrized by $t/s=:z$ on the complex plane, up to measure zero. The Fubini-Study form evaluates to ${i\over 2\pi}{dz\wedge d\bar z\over(1+|z|^2)^2}$ on $\C\P^1$. Direct integration gives \[ \int_\C \frac{i}{2\pi}\frac{dz\wedge d\bar z}{(1+|z|^2)^2}=\frac{1}{\pi}\int_{\R^2}\frac{dvol}{(1+r^2)^2}=\int_0^\infty dr^2/(1+r^2)^2=-1/(1+r^2)|^\infty_0=1\]
	\end{proof}

	%
	
	
	Recall that the point in $\Delta_{2,4}$ is given by $x=(x_1,x_2,x_3,x_4)$ with $x_1+x_2+x_3+x_4=2$, so $\bar x=(x_1,x_2,x_3)$ gives a coordinate system on $\Delta_{2,4}^\circ$.
	Slightly abusing the notation, due to Proposition \ref{prop_E(gamma_i)}, we may write
	\[
	x_1=E(\gamma_1) \ , \ x_2=E(\gamma_2) \ , \ x_3 =E(\gamma_3)
	\]
	as real-valued functions on $B=\Delta_{2,4}^\circ \times \mathbb R$.
	We also introduce:
	\begin{equation}
		\label{psi_pm_eq}
		\psi_+ = E(\a_{34}^+) : U_+ \to \mathbb R  \ , \qquad  
		\psi_-  = E(\b_{34,24}^-) :U_-  \to \mathbb R .
	\end{equation}
	Recall that $\a_{34}^+$ and $\b_{34,24}^-$ are sections in $\mathscr R(U_+)$ and $\mathscr R (U_-)$, respectively, defined in \eqref{alpha_beta_+_eq} and \eqref{alpha_beta_-_eq} by the following intersection pattern
	\begin{table}[H]
		\centering
		\begin{tabular}{c|c|c|c|c|c|c}
			& $D_{13}$  &  $D_{24}$ & $D_{12}$ & $D_{23}$& $D_{34}$ & $D_{14}$ \\\hline 
			$\alpha_{34}^+$     &   $0$&   $0$&   $0$&   $0$& $1$ &   $0$\\\hline
			$\beta_{34,24}^-$     &   $0$&   $1$&   $0$&   $0$& $1$ &   $0$\\\hline
		\end{tabular}
	\end{table}

	Observe that $U_+\cap U_-=\mathcal N_{13}\sqcup \mathcal N_{24}$ has two connected components.
	By applying the energy to the relation \eqref{eq_disk_monodromy}, we obtain that
	$	\psi_+ = \psi_-$ on $\mathcal N_{24}$ and $\psi_+=\psi_-+x_1+x_3-E(\mathcal H)$ on $\mathcal N_{13}$.
	Equivalently, on their common domains, we have
	\begin{align}
		\psi_+=\psi_-+\min\{0,x_1+x_3-1\}  \label{eq_A_intaff}
	\end{align}

	Define
	\begin{equation}
		\label{eq_chi_+}
		\chi_+ = ( x_1 \ , \ x_2 \ , \ x_3 \ , \ E(\alpha_{34}^+) ) = (x_1 \ , \ x_2 \ , \ x_3 \ , \psi_+( x, \rho))  \ : \ U_+ \to V_+ 
	\end{equation}
	\[
	\chi_- = (x_1 \ , \ x_2 \ , \ x_3 \ , \ E(\beta_{34,24}^-) ) = (x_1 \ , \ x_2 \ , \ x_3 \ , \ \psi_-( x, \rho) ) \ : \  U_- \to V_-
	\]
	where $V_\pm=\chi_\pm(U_\pm)\subset \mathbb R^4$.
	Since $U_+$ and $U_-$ are contractible, it follows from \cite[Theorem 2.2]{action_angle} that the map $\chi_+$ is a local diffeomorphism (not necessarily a diffeomorphism at this moment), and similarly for $\chi_-$.

	The following lemma aims to show that they are indeed diffeomorphisms.
	Introduce a continuous function $\psi$ on $B$ defined by 
	\begin{equation}
		\label{eq_psi_defn}
		\psi(x,\r)=\begin{cases}
			\psi_-(x,\r)+\min\{0,x_1+x_3-1\}, \quad & \text{on} \ U_-\\
			\psi_+(x,\r),\quad & \text{on} \ U_+
		\end{cases}
	\end{equation}
	Remark that the function $\psi(x,\rho)$ can be represented by the symplectic area of some topological disk bounded by the Lagrangian fiber over $(x,\rho)\in B$. The presence of the Lagrangian fibration then allows one to extend continuously this function over the singular locus $\Gamma = \Pi \times \{0\}$ (Proposition \ref{singular_locus_prop}), obtaining a continuous function defined on the entire $B$, still denoted as
	\[
	\psi: B\to \mathbb R
	\]

	\begin{lem}
		\label{decreasing_lem}
		For each fixed $x\in \Delta_{2,4}$, the function $\rho\mapsto \psi(x,\rho)$ is strictly decreasing.
		In particular, $\chi_+$ and $\chi_-$ are diffeomorphisms.
		\label{lem_monotonic}
	\end{lem}
	\begin{proof}
		Recall that $\chi_+$ is a local diffeomorphism and $U_+$ is contractible and connected. Then, the tangent map $d\chi_{+}:T_pU_+\rightarrow T_{\chi_+(p)}V_+$ is a linear isomorphism at any point $p\in U_+$, and its determinant is nowhere vanishing on $U_+$.
		In particular, the determinant is either always positive or always negative.
		Observe that the Jacobian of $d\chi_+$ has the form
		\[
		\begin{pmatrix}
			1 & 0 & 0 & \partial_{x_1} \psi_+ \\
			0 & 1 & 0 & \partial_{x_2} \psi_+ \\
			0 & 0 & 1 & \partial_{x_3} \psi_+ \\
			0 & 0 & 0 & \partial_{\rho} \psi_+ 
		\end{pmatrix}
		\]
		and the determinant is simply $\partial_\r \psi_+$.
		The following computation suggests that $\partial_\r \psi_+$ is negative at a specific point and thus negative on the entire $U_+$. The same conclusion holds on $U_-$.
		
		Choose the point $x_0=(1/2,1/2,1/2)$, and let's compute
		$\r\mapsto \psi_+(x_0,\r)$ explicitly as follows.
		Recall the explicit holomorphic disk 
		\[v_{34}^+ =v_{34}^+(R):\mathbb D\rightarrow \C\P^5 \ , \qquad  z\mapsto 
		[Z_{13}: Z_{24}: Z_{12}: Z_{23}: Z_{34}: Z_{14}] = 
		\left[\sqrt{R^2+z}:\sqrt{R^2+z}:1:R:z:R \right]
		\]
		introduced in Section \ref{s_explicit_disks} where we set $\r=2\log R$ and $\mathbb D$ is the unit disk in $\mathbb C$. It represents the class $\a_{34}^+$. 
		
		Notice that the disk $v_{34}^+$ is contained in the affine chart $Z_{12}\neq 0$ where we put $z_{ij}=Z_{ij}/Z_{12}$.
		Then, the Fubini-Study form is given by 
		\[
		\w=\frac{\mathbf i}{2\pi}\partial \bar \partial \log(1+ |z_{13}|^2+|z_{24}|^2+|z_{23}|^2+|z_{34}|^2+|z_{14}|^2) 
		\]
		and $\omega=d\lambda$ where we take the primitive 
		\begin{align*}
			\lambda=&-\frac{\mathbf i}{2\pi} \partial \log (1+ |z_{34}|^2+|z_{13}|^2+ |z_{14}|^2+ |z_{23}|^2+|z_{24}|^2) \\
			=&-\frac{\mathbf i}{2\pi} \frac{\bar z_{34}dz_{34}+\bar z_{13}dz_{13}+\bar z_{14}dz_{14}+\bar z_{23}dz_{23}+\bar z_{24}dz_{24}}{1+ |z_{34}|^2+|z_{13}|^2+ |z_{14}|^2+ |z_{23}|^2+|z_{24}|^2}
		\end{align*}
		
		Since $\rho=2\log R$, the function $\rho\mapsto \psi_+(x_0,\rho)$ can be regarded as a function of $R$, denoted by 
		\[A_+(R)=\int_{\a_{34}^+}\w=\int_{\mathbb D}{v_{34}^+}^\ast \w=\int_{\partial \mathbb D}{v_{34}^+}^\ast \lambda\]
		We can give a parametrization $\gamma$ of the restriction of $v_{34}^+$ on $S^1\cong \partial\mathbb D$ as follows: for $t\in [0,1]$, we set
		\[
		\gamma(t) = (z_{13}(t), z_{24}(t), z_{23}(t), z_{34}(t), z_{14}(t)) = (\sqrt{R^2+ e^{2\pi \mathbf i t}} \ , \ \sqrt{R^2+ e^{2\pi \mathbf i t}} \ , \ R \ , \ e^{2\pi \mathbf i t} \ , \ R ) 
		\]
		Compute
		\begin{align*}
			\dot \gamma(t)=2\pi \mathbf i e^{2\pi\mathbf i t} \left(\frac{1}{2\sqrt{R^2+ e^{2\pi \mathbf i t}}}\partial_{z_{13}}+\frac{1}{2\sqrt{R^2+ e^{2\pi \mathbf i t}}}\partial_{z_{24}}+ \partial_{z_{34}} \right)
		\end{align*}
		Then,
		\begin{align*}
			A_+(R) & = \int_0^1 \lambda (\dot \gamma(t)) dt \\
			&= \int_0^1\frac{\frac{R^2e^{2\pi \mathbf i t}+1}{|R^2+e^{2\pi \mathbf i t}|}+1}{2+2|R^2+e^{2\pi \mathbf i t}|+2R^2}dt=\int_0^1\left(\frac{1}{4}-\frac{R^2-1}{4|R^2+e^{2\pi\mathbf i t}|}\right)dt\label{eq_disk1_area}
		\end{align*}
		This is an elliptic integral. 
		One can compute its derivative with respect to $R$:
		\[{dA_+ \over dR}=-\frac{R}{2}\int_0^1 \frac{(1+R^2)(1+\cos 2\pi \mathbf i t)}{|R^2+e^{2\pi \mathbf i t}|^3}dt <0 \]
		Hence the function $\rho\mapsto \psi_+(x_0,\rho)$ is strictly decreasing in $\rho$.
		
		On the other hand, one can represent the disk $\b_{34,24}^-$ by the map
		$u_{34,24}^-:\mathbb D^2\rightarrow \C\P^5, z\mapsto [\sqrt{R^2+z}:\frac{1+R^2z}{z+R^2}\sqrt{R^2+z}:R:1:Rz:1]$. The disk is contained in the affine chart $Z_{23}\neq 0$, and bounded by the fiber at the point $(x,\r)=(x_0,-2\log R)$. In the same way we define the following disk area function representing $\psi_-$ 
		\[
		A_-(R):=\int_{\b_{34,24}^-}\w=\int_{\partial \mathbb D}{u_{34,24}^-}^\ast\l \]
		Explicit computation yields
		\begin{align}
			A_-(R)=\int_0^1{R^2\frac{(R^2+e^{2\pi \mathbf i t})}{|R^2+e^{2\pi \mathbf i t}|}+R^2\over 2+2R^2+2|R^2+e^{2\pi \mathbf i t}|}dt=\int_0^1\left(\frac{1}{4}+\frac{R^2-1}{4|R^2+e^{2\pi\mathbf i t}|}\right)dt
		\end{align}
		Similarly, one can show that $A_-$ is strictly increasing with respect to $R$.
		Since $\rho=-2\log R$, the function $\r\mapsto \psi_-(x_0,\r)$ is strictly decreasing. 
	\end{proof}

	Now, by \eqref{eq_A_intaff} and \eqref{eq_chi_+}, the transition map 
	$
	\chi_+ \circ \chi_-^{-1}
	$
	is integral affine as it is of the form 
	\[
	(x_1,x_2,x_3,x)\mapsto (x_1,x_2,x_3,x +\min\{0,x_1+x_3-1\} )
	\]
	Therefore, we conclude the following result:
	
	\begin{thm}
		\label{thm_integral_atlas}
		The smooth locus $B_0$ is an integral affine manifold, equipped with an integral affine atlas $\{(U_+,\chi_+), (U_-,\chi_-) \}$.
	\end{thm}

	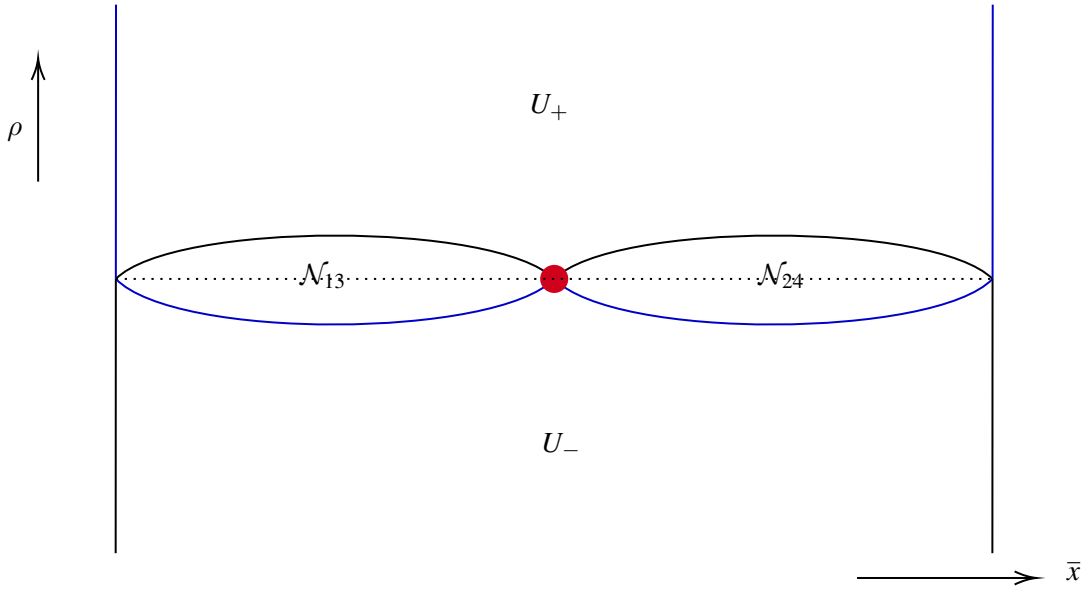
\begin{figure}
		\centering
		
		\tikzset{every picture/.style={line width=0.75pt}} 
		
		\begin{tikzpicture}[x=0.75pt,y=0.75pt,yscale=-1,xscale=1]
			
			\draw    (320.09,140.35) .. controls (349.09,111.35) and (511.09,111.35) .. (540.09,140.35) ;
			\draw    (100.09,140.35) .. controls (129.09,111.35) and (291.09,111.35) .. (320.09,140.35) ;
			\draw [color={rgb, 255:red, 0; green, 0; blue, 200 }  ,draw opacity=1 ]   (100.09,140.35) .. controls (130.09,171.35) and (291.09,170.35) .. (320.09,140.35) ;
			\draw [color={rgb, 255:red, 0; green, 0; blue, 200 }  ,draw opacity=1 ]   (320.09,140.35) .. controls (350.09,171.35) and (511.09,170.35) .. (540.09,140.35) ;
			\draw  [draw opacity=0][fill={rgb, 255:red, 208; green, 2; blue, 27 }  ,fill opacity=1 ] (313.05,140.35) .. controls (313.05,136.46) and (316.2,133.31) .. (320.09,133.31) .. controls (323.98,133.31) and (327.14,136.46) .. (327.14,140.35) .. controls (327.14,144.24) and (323.98,147.4) .. (320.09,147.4) .. controls (316.2,147.4) and (313.05,144.24) .. (313.05,140.35) -- cycle ;
			\draw [color={rgb, 255:red, 0; green, 0; blue, 200 }  ,draw opacity=1 ]   (100.18,2.7) -- (100.09,140.35) ;
			\draw [color={rgb, 255:red, 0; green, 0; blue, 200 }  ,draw opacity=1 ]   (540.18,2.7) -- (540.09,140.35) ;
			\draw [color={rgb, 255:red, 0; green, 0; blue, 0 }  ,draw opacity=1 ]   (100.09,140.35) -- (100,278) ;
			\draw [color={rgb, 255:red, 0; green, 0; blue, 0 }  ,draw opacity=1 ]   (540.09,140.35) -- (540,278) ;
			\draw  [dash pattern={on 0.84pt off 2.51pt}]  (100.09,140.35) -- (540.09,140.35) ;
			\draw    (61.09,91.53) -- (61.09,31.53) ;
			\draw [shift={(61.09,29.53)}, rotate = 90] [color={rgb, 255:red, 0; green, 0; blue, 0 }  ][line width=0.75]    (10.93,-3.29) .. controls (6.95,-1.4) and (3.31,-0.3) .. (0,0) .. controls (3.31,0.3) and (6.95,1.4) .. (10.93,3.29)   ;
			\draw    (472.09,290.44) -- (560.09,290.44) ;
			\draw [shift={(562.09,290.44)}, rotate = 180] [color={rgb, 255:red, 0; green, 0; blue, 0 }  ][line width=0.75]    (10.93,-3.29) .. controls (6.95,-1.4) and (3.31,-0.3) .. (0,0) .. controls (3.31,0.3) and (6.95,1.4) .. (10.93,3.29)   ;
			
			\draw (43,60.4) node [anchor=north west][inner sep=0.75pt]    {$\rho $};
			\draw (190,131.4) node [anchor=north west][inner sep=0.75pt]    {$\mathcal N_{13}$};
			\draw (420,131.4) node [anchor=north west][inner sep=0.75pt]    {$\mathcal N_{24}$};
			\draw (305,46.4) node [anchor=north west][inner sep=0.75pt]    {$U_{+}$};
			\draw (311,216.4) node [anchor=north west][inner sep=0.75pt]    {$U_{-}$};
			\draw (575,280.4) node [anchor=north west][inner sep=0.75pt]    {$\overline{x}$};

		\end{tikzpicture}
		
		\caption{
			\footnotesize A schematic picture for the base. The $\bar x$ direction is compressed: it is more precisely the octahedron as in Figure \ref{fig:wall}, and now abbreviated by distance to the blue region. The dashed line denotes the "wall", and the red dot denotes the singular locus, which we understand to be the blue region in the wall.}
		\label{fig:schematic base}
	\end{figure}

	\section{B-side}

	\subsection{Family Floer mirror construction} \label{s_family_floer_mirror}
	
	Here we carry out the constructions of the mirror space and the mirror torus fibration explicitly, without providing detailed explanations of the motivations behind them. 
	There are indeed deep Floer-theoretic reasons for these constructions, but the procedures themselves are explicit and transparent, requiring no prior knowledge of Floer theory.
	Thus, while we use the term ``family Floer mirror,'' one should for the moment regard our mirror construction as independent of any Floer-theoretic machinery.

	\subsubsection{Flat unitary local systems}\label{sssec_ydef}
	Let
	$
	\Lambda=\mathbb C((T^{\mathbb R}))
	$
	denote the \textit{Novikov field} consisting of formal power series
	\[
	x = \sum_{i=0}^\infty a_i T^{\lambda_i}
	\]
	with $a_i\in\mathbb C$ and $\lambda_i \nearrow \infty$.
	Here $T$ is a formal symbol.
	It admits a non-archimedean valuation map $\val(x)= \min\{ \lambda_i \mid a_i\neq 0\}$ and an induced norm $|x|=e^{-\val(x)}$. The unit circle in $\Lambda$ is denoted by 
	\[
	U_\Lambda:=\{ x\in\Lambda\mid |x|=1 \}
	\]
	
	Observe that the identification 
	\[
	H^1(L_q; U_\Lambda) \cong \mathrm{Hom}(\pi_1(L_q), U_\Lambda)
	\]
	gives a natural pairing, written as
	\[
	H^1(L_q; U_\Lambda) \times \pi_1(L_q) \to U_\Lambda \quad , \quad (\mathbf y, \sigma)\mapsto \mathbf y^{\sigma}
	\]
	
	Recall that the tropicalization map 
	\[\trop: (\Lambda^*)^4\to\mathbb R^4
	\]
	is a continuous map with respect to the analytic topology on $(\Lambda^*)^4$ and the Euclidean topology on $\mathbb R^4$; see Section \ref{s_trop_affinoid_torus_fib}.
	For each open subset $V\subset \mathbb R^4$, we have an analytic open domain $\trop^{-1}(V)\subset (\Lambda^*)^4$.

	With the aforementioned integral affine atlas $\{(U_+,\chi_+), (U_-,\chi_-)\}$ around (\ref{eq_chi_+}), we consider a bijective map $\tilde \chi_+$ over $\chi_+$
	\[
	\xymatrix{\displaystyle \bigcup_{q\in U_+} H^1(L_q; U_\Lambda) \ar[d] \ar[rr]^{\tilde \chi_+} & &  \trop^{-1}(V_+) \ar[d] \\
		U_+ \ar[rr]^{\chi_+} & & V_+ 
	}
	\]
	defined by
	\[
	(q,\mathbf y)  \  \mapsto  \  (y^+_1,y^+_2,y^+_3, y^+) = \left( T^{E(\gamma_1)(q)} \mathbf y^{\partial \gamma_1(q)}  \ , \ T^{E(\gamma_2)(q)} \mathbf y^{\partial \gamma_2(q)}
	\ , \ 
	T^{E(\gamma_3)(q)} \mathbf y^{\partial \gamma_3(q)} \ , \ T^{E(\alpha^+_{34})(q)} \mathbf y^{\partial \alpha^+_{34}(q)}
	\right)
	\]
	where $q\in U_+$, $\mathbf y\in H^1(L_q; U_\Lambda)$, and we write $y_1,y_2,y_3, y_+$ for coordinate functions in $\trop^{-1}(V_+)\subseteq (\Lambda^*)^4$. Similarly, for $\chi_-$, we can define
	\[
	\bigcup_{q\in U_-} H^1(L_q; U_\Lambda) \xlongrightarrow{\tilde \chi_-} \trop^{-1}(V_-)
	\]
	over $\chi_-$ through the analogous formula
	\[
	(y_1^-,y_2^-,y_3^-,y^-) = \left( T^{E(\gamma_1)(q)} \mathbf y^{\partial \gamma_1(q)}  \ , \ T^{E(\gamma_2)(q)} \mathbf y^{\partial \gamma_2(q)}
	\ , \ 
	T^{E(\gamma_3)(q)} \mathbf y^{\partial \gamma_3(q)} \ , \ T^{E(\beta^-_{34,24})(q)} \mathbf y^{\partial \beta^-_{34,24} (q)}
	\right)
	\]
	
	In the above situation, if we have an analytic isomorphism $\varphi$ over $\chi_-\circ \chi_+^{-1}$ as follows
	\[
	\xymatrix{
		\trop^{-1}( \chi_+(U_+\cap U_-) ) \ar[rr]^\varphi \ar[d] & &  \trop^{-1}(\chi_-(U_+\cap U_-)) \ar[d] \\
		\chi_+(U_+\cap U_-)  \ar[rr]^{\chi_-\circ \chi_+^{-1}} & & \chi_-(U_+\cap U_-)
	}
	\]
	\begin{figure}
		\centering

		\tikzset{every picture/.style={line width=0.75pt}} 
		
		\begin{tikzpicture}[x=0.75pt,y=0.75pt,yscale=-1,xscale=1]
			
			\draw  [color={rgb, 255:red, 0; green, 0; blue, 0 }  ,draw opacity=0.41 ][fill={rgb, 255:red, 243; green, 215; blue, 215 }  ,fill opacity=0.42 ] (60.31,112.75) .. controls (60.31,80.86) and (86.17,55) .. (118.06,55) .. controls (149.96,55) and (175.81,80.86) .. (175.81,112.75) .. controls (175.81,144.64) and (149.96,170.5) .. (118.06,170.5) .. controls (86.17,170.5) and (60.31,144.64) .. (60.31,112.75) -- cycle ;
			\draw  [color={rgb, 255:red, 0; green, 0; blue, 0 }  ,draw opacity=0.46 ][fill={rgb, 255:red, 172; green, 189; blue, 223 }  ,fill opacity=0.42 ] (60.31,191.75) .. controls (60.31,159.86) and (86.17,134) .. (118.06,134) .. controls (149.96,134) and (175.81,159.86) .. (175.81,191.75) .. controls (175.81,223.64) and (149.96,249.5) .. (118.06,249.5) .. controls (86.17,249.5) and (60.31,223.64) .. (60.31,191.75) -- cycle ;
			\draw    (185,94) .. controls (207.47,68.08) and (264.08,25.97) .. (303.04,40.96) ;
			\draw [shift={(304.81,41.69)}, rotate = 203.55] [color={rgb, 255:red, 0; green, 0; blue, 0 }  ][line width=0.75]    (10.93,-3.29) .. controls (6.95,-1.4) and (3.31,-0.3) .. (0,0) .. controls (3.31,0.3) and (6.95,1.4) .. (10.93,3.29)   ;
			\draw    (189,206) .. controls (215.41,231.3) and (261.41,264.67) .. (301.96,254.2) ;
			\draw [shift={(303.81,253.69)}, rotate = 163.69] [color={rgb, 255:red, 0; green, 0; blue, 0 }  ][line width=0.75]    (10.93,-3.29) .. controls (6.95,-1.4) and (3.31,-0.3) .. (0,0) .. controls (3.31,0.3) and (6.95,1.4) .. (10.93,3.29)   ;
			\draw  [dash pattern={on 0.84pt off 2.51pt}] (296,2) -- (432.81,2) -- (432.81,124.69) -- (296,124.69) -- cycle ;
			\draw  [dash pattern={on 0.84pt off 2.51pt}] (296,164) -- (432.81,164) -- (432.81,286.69) -- (296,286.69) -- cycle ;
			\draw  [fill={rgb, 255:red, 243; green, 215; blue, 215 }  ,fill opacity=0.45 ] (318,37.14) .. controls (318,27.12) and (326.12,19) .. (336.14,19) -- (393.68,19) .. controls (403.69,19) and (411.81,27.12) .. (411.81,37.14) -- (411.81,91.55) .. controls (411.81,101.57) and (403.69,109.69) .. (393.68,109.69) -- (336.14,109.69) .. controls (326.12,109.69) and (318,101.57) .. (318,91.55) -- cycle ;
			\draw  [fill={rgb, 255:red, 172; green, 189; blue, 223 }  ,fill opacity=0.39 ] (317.5,198.14) .. controls (317.5,188.12) and (325.62,180) .. (335.64,180) -- (393.18,180) .. controls (403.19,180) and (411.31,188.12) .. (411.31,198.14) -- (411.31,252.55) .. controls (411.31,262.57) and (403.19,270.69) .. (393.18,270.69) -- (335.64,270.69) .. controls (325.62,270.69) and (317.5,262.57) .. (317.5,252.55) -- cycle ;
			\draw    (320,188) .. controls (352.81,206.69) and (373.81,207.69) .. (407.81,188.69) ;
			\draw    (321,100.55) .. controls (351.81,81.69) and (378.81,82.69) .. (408.81,101.24) ;
			\draw    (356,108) .. controls (379.22,119.4) and (382.65,154.85) .. (372.61,177.94) ;
			\draw [shift={(371.81,179.69)}, rotate = 295.56] [color={rgb, 255:red, 0; green, 0; blue, 0 }  ][line width=0.75]    (10.93,-3.29) .. controls (6.95,-1.4) and (3.31,-0.3) .. (0,0) .. controls (3.31,0.3) and (6.95,1.4) .. (10.93,3.29)   ;
			
			\draw (105,200) node [anchor=north west][inner sep=0.75pt]   [align=left] {$\displaystyle U_{-}$ };
			\draw (104,89) node [anchor=north west][inner sep=0.75pt]   [align=left] {$\displaystyle U_{+}$ };
			\draw (123.41,158.34) node  [font=\footnotesize] [align=left] {\begin{minipage}[lt]{49.51pt}\setlength\topsep{0pt}
					$\displaystyle U_{+} \cap \ U_{-}$
			\end{minipage}};
			\draw (224,70) node [anchor=north west][inner sep=0.75pt]   [align=left] {$\displaystyle \chi _{+}$};
			\draw (227,208) node [anchor=north west][inner sep=0.75pt]   [align=left] {$\displaystyle \chi _{-}$};
			\draw (339,30) node [anchor=north west][inner sep=0.75pt]   [align=left] {$\displaystyle \chi _{+}( U_{+})$};
			\draw (336,236) node [anchor=north west][inner sep=0.75pt]   [align=left] {$\displaystyle \chi _{-}( U_{-})$};
			\draw  [draw opacity=0]  (363.5, 101) circle [x radius= 56, y radius= 13.35]   ;
			\draw (325.08,93.57) node [anchor=north west][inner sep=0.75pt]  [font=\footnotesize,rotate=-0.64] [align=left] {$\displaystyle \chi _{+}( U_{+} \cap U_{-})$};
			\draw (328,180) node [anchor=north west][inner sep=0.75pt]  [font=\footnotesize] [align=left] {$\displaystyle \chi _{-}( U_{+} \cap U_{-})$};
			\draw (380,134) node [anchor=north west][inner sep=0.75pt]  [font=\small] [align=left] {$\displaystyle \chi _{-} \circ \chi _{+}^{-1}$};
			\draw (409,4) node [anchor=north west][inner sep=0.75pt]   [align=left] {$\displaystyle \mathbb{R}^{4}$};
			\draw (409,265) node [anchor=north west][inner sep=0.75pt]   [align=left] {$\displaystyle \mathbb{R}^{4}$};

		\end{tikzpicture}
		
		\label{fig:chi_pm_transition}
	\end{figure}

	then we can equip the set $\bigcup_{q\in B_0} H^1(L_q; U_\Lambda)$ with an analytic space structure given by
	\[
	\trop^{-1}(V_+) \cup \trop^{-1}(V_-) / \!\sim_\varphi
	\]
	where $\sim_\varphi$ denotes the gluing along $\varphi$.
	Note that $\chi_+(U_+\cap U_-)$ and $\chi_-(U_+\cap U_-)$ are open subsets of $V_+$ and $V_-$ respectively.
	
	In general, such a $\varphi$ need not exist, and even when it does, it need not be unique.
	However, the family Floer SYZ mirror construction developed in \cite{Yuan_I_FamilyFloer} offers a canonical algorithm of such a $\varphi$ that roughly captures the Floer theory of Lagrangian fibers $L_q$'s.
	Let's discuss this in the following section.
	
	\subsubsection{Quantum-corrected analytic gluing formula}
	
	
	According to our mirror construction algorithm, we need to determine which holomorphic disks exist in a given topological class in $\pi_2(X,L_q)$ of Maslov index $2$, or equivalently, which topologically defined disks are actually realized by holomorphic disks bounded by the fibers. 
	Identifying such holomorphic disks may be difficult for general fibers, but it is feasible for certain fixed ones. 
	The argument of Lagrangian isotopy then allows one to transport these holomorphic disks to other fibers, provided no Maslov index zero disks appear. 
	We remark that the critical values of the B-side mirror superpotential are expected to be eigenvalues of quantum multiplication by the first Chern class $c_1$ on the A-side, regardless of whether Maslov-zero quantum corrections are present; see \cite{Yuan_c_1}.

	Nevertheless, to keep the discussion straightforward, the following construction is presented in a manner entirely independent of any Floer-theoretic machinery.
	
	\begin{construction}
		\label{construction_W}
		Given a Lagrangian fiber $L_q$, we define a set-theoretic function
		$
		H^1(L_q; U_\Lambda) \to \Lambda
		$
		\[
		W(\mathbf y)=\sum_{\beta} T^{E(\beta)} \mathbf y^{\partial\beta}
		\]
		where $\beta$ runs through all elements in $\pi_2(\mathbb X, L_q)$ 
		such that $\mu(\beta)=2$ and $\beta$ can be represented by a holomorphic disk $u:(\mathbb D,\partial\mathbb D)\to (\mathbb X, L_q)$ with $\bar\partial u=0$.
		Note that here $\mu(\beta)$ is the Maslov index of $\beta\in \pi_2(X, L_q)$, which is equal to $2\beta\cdot D_{ac}$ by \cite[Lemma 3.1]{AuTDual}.
	\end{construction}
	
	In Section \ref{s_local_system_topo},
	we have already constructed six holomorphic disks in $U_\pm^{\prime}$ bounded by the specific Lagrangian fibers over the two points
	$q_+=(x_0, 2\log R) \in U_+$ and $q_-=(x_0, -2\log R)\in U_-$;
	see (\ref{alpha_beta_+_eq}) and (\ref{alpha_beta_-_eq}).

	Applying the formula in Construction \ref{construction_W} yields that for $\mathbf y\in H^1(L_{q_+}; U_\Lambda)$,
	\begin{align*}
		W(\mathbf y)= T^{E(\alpha^+_{12})} \mathbf y^{\partial \alpha^+_{12}} 
		+  T^{E(\alpha^+_{34})} \mathbf y^{\partial \alpha^+_{34}}
		+  T^{E(\beta^+_{23,13})} \mathbf y^{\partial \beta^+_{23,13} }
		+  T^{E(\beta^+_{23,24})} \mathbf y^{\partial \beta^+_{23,24} }
		+  T^{E(\beta^+_{14,13})} \mathbf y^{\partial \beta^+_{14,13} }
		+  T^{E(\beta^+_{14,24})} \mathbf y^{\partial \beta^+_{14,24} }
	\end{align*}
	Further using the relation \eqref{eq_alpha_to_gamma_H} and the aforementioned identification maps $\tilde \chi_+$, we produce an analytic map
	\[
	W_+: \trop^{-1}(V_+) \to \Lambda
	\]
	given by
	\[
	W_+=\frac{y^+_1y^+_2y^+}{T}+\frac{y^+_3}{y^+}+\frac{T}{y^+_1y^+}+y^++\frac{T}{y^+_2y^+}+\frac{T^{2}}{y^+_1y^+_2y^+_3y^+}
	\]
	Similarly, we obtain an analytic map 
	$$W_-:\trop^{-1}(V_-)\to \L$$
	given by 
	\[ 
	W_-=\frac{(y^-_1)^2y_2^-y_3^-y^-}{T^2}+\frac{y^-_1y^-_2y^-}{T}+\frac{T}{y^-_1y^-}+\frac{y_1^-y_3^-y^-}{T}+y^-+\frac{T^2}{y^-y_1^-y_2^-y_3^-}
	\]

	For convenience we also introduce: 
	\begin{align*}
		&\begin{cases}
			W_+^{12}&=T^{-1}y^+_1y^+_2y^+\\
			W^{23}_+&=\frac{y^+_3}{y^+}+\frac{T}{y^+_1y^+}\\
			W^{34}_+&=y^+\\
			W^{41}_+&=\frac{T}{y^+_2y^+}+\frac{T^{2}}{y^+_1y^+_2y^+_3y^+}
		\end{cases}\\
		&\begin{cases}
			W^{12}_-&=T^{-2}(y^-_1)^2y_2^-y_3^-y^-+T^{-1}y^-_1y^-_2y^-\\
			W^{23}_-&=\frac{T}{y^-_1y^-}\\
			W^{34}_-&=T^{-1}y_1^-y_3^-y^-+y^-\\
			W^{41}_-&=\frac{T^2}{y^-y_1^-y_2^-y_3^-}
		\end{cases}
	\end{align*}
	collecting those terms that correspond to the disks with nontrivial intersections with $D_{12}, D_{23}, D_{34}, D_{41}$ respectively.
	For instance, $W^{23}_+$ contains monomials that correspond to holomorphic disks that intersect only with divisors $D_{23}$, and so on.
	
	\subsubsection{Gluing of two local charts}
	
	Now, we aim to construct the mirror analytic space
	\begin{equation}
		\label{glue_two_charts_eq}
		X^\vee_0 := \trop^{-1}(V_+) \cup \trop^{-1}(V_-) /\sim_\varphi
	\end{equation}
	by gluing the previous two local charts.
	Here the gluing map 
	\[
	\trop^{-1}(V_+) \supseteq 
	\trop^{-1}( \chi_+(U_+\cap U_-) )  \xlongrightarrow{\varphi}  \trop^{-1}(\chi_-(U_+\cap U_-)) 
	\subseteq \trop^{-1}(V_-)
	\]
	is to be determined.
	We require that
	\[
	\varphi^\ast y^-_i=y^+_i  \ , \quad \varphi^\ast W_-  = W_+.
	\]
	Thus $y_i^+$ and $y_i^-$ give rise to a globally defined invertible analytic function on $X_0^\vee$, which we denote by 
	\begin{equation}
		\label{eq_y_i}
		y_i: X_0^\vee \to \Lambda
	\end{equation}
	We also abbreviate $(y_1,y_2,y_3)$ as $\bar y$.
	Here the first condition reflects the $T^3$-symmetry of the Lagrangian torus fibers, while the second reflects the general principle that the analytic gluing maps in the family Floer mirror construction must preserve the potential functions.
	
	Accordingly, we have
	\begin{equation}
		\label{phi_on_y_+}
		\varphi^\ast y^-=\frac{y^+}{1+T^{-1}y_1y_3}
	\end{equation}
	In this setting, $W_+$ and $W_-$ glue to an analytic function 
	\[
	W: X_0^\vee \to \Lambda
	\]
	Moreover, one actually sees that each pair $W_+^{i,i+1}$ and $W_-^{i,i+1}$ can also be glued in the sense that $\varphi^\ast W^{i,i+1}_-=W^{i,i+1}_+$, obtaining analytic functions 
	\[
	W^{i,i+1}: X_0^\vee \to \Lambda
	\]
	Observe that
	\begin{equation}
		\label{W_W_ii+1_eq}
		W=W^{12}+W^{23}+W^{34}+W^{41}
	\end{equation}

	


	\subsection{Explicit description of the mirror space and fibration} \label{s_explicit_mirror_fibration}

	A main purpose of this paper is to propose that, when the B-side mirror is formulated over the Novikov field, one can realize the above mirror correspondence precisely and explicitly within the SYZ framework.
	Indeed, the mirror of $\mathbb X=Gr(2,4)$ is expected to be itself (removing the anti-canonical divisor), denoted by $$Y= Gr(2,4) \setminus D_{ac},$$ which is further equipped with a Landau-Ginzburg superpotential $W$ on $Y$.
	
	So far, we have a mirror analytic space $X_0^\vee=\trop^{-1}(V_+)\cup \trop^{-1}(V_-)/\sim_\varphi$ as in \eqref{glue_two_charts_eq} and a mirror affinoid torus fibration $\pi_0^\vee: X_0^\vee\to B_0$ obtained by gluing two trivial fibrations $\trop^{-1}(V_\pm)\to V_\pm$ and identifying $V_\pm$ with $U_\pm$, respectively, through $\chi_{\pm}$.
	We also remark that despite the underlying Floer-theoretic foundation, the gluing map $\varphi$ has an explicit form \eqref{phi_on_y_+}.

	The goal of this section is to make the abstract mirror space and fibration into a more explicit form. Specifically, we aim to build a commutative diagram
	\begin{align}
		\xymatrix{
			X_0^\vee\ar[rr]^{g} \ar[d]^{\pi_0^\vee} & & Y \ar[d]^{f}\\
			B_0\ar[rr]^{j} & & \R^5
		} \label{commd} 
	\end{align}
	where
	
	$\bullet$ $g$ is a morphism of analytic spaces over the Novikov field.
	
	$\bullet$ $j:B\to \mathbb R^5$ is a topological homeomorphism onto its image which in some sense unfolds the integral affine action coordinates in $B_0$.
	
	$\bullet$ $f$ is a continuous map with respect to the non-archimedean analytic topology and the Euclidean topology in the base.
	
	We will construct them as follows.

	\subsubsection{The analytic map $g$}

	Note that part of our mirror space candidate is obtained by gluing two charts $\trop^{-1}(V_\pm) \subset (\Lambda^*)^4$. 
	It is reminiscent of the known fact that $Gr(2,4)$, over any field $\Bbbk$, admits two cluster charts, each isomorphic to $(\Bbbk^*)^4$. 
	This observation suggests that the glued space 
	\[
	X_0^\vee = \trop^{-1}(V_+) \cup \trop^{-1}(V_-) / \!\sim_\varphi
	\]
	may embed into $Gr(2,4)$ in an appropriate way so that $\trop^{-1}(V_+)$ is embedded into one cluster chart and $\trop^{-1}(V_-)$ is embedded into another.


	Let
	$$[p_{ij}]=[p_{13}:p_{24}:p_{12}:p_{23}:p_{34}:p_{41}]$$
	denote Pl\"ucker coordinates on the Grassmannian $Gr(2,4)$ over $\L$. 
	We then construct the analytic maps
	\begin{align*}
		g_+: \trop^{-1}(V_+) & \to Gr(2,4)_{\Lambda} \qquad 
		(y^+_1,y^+_2,y^+_3,y^+) &&\mapsto \left[\left(\frac{y^+_1y^+_3+T}{y^+}\right):y^+_1y^+_2y^+:y^+_1:y^+_1y^+_2:y^+_1y^+_2y^+_3:T\right] \\
		g_-: \trop^{-1}(V_-) & \to Gr(2,4)_{\Lambda} \qquad
		(y^-_1,y^-_2,y^-_3,y^-) && \mapsto \left[\frac{T}{y^-}:y^-_1y^-_2y^-(T^{-1}y^-_1y^-_3+1):y^-_1:y^-_1y^-_2:y^-_1y^-_2y^-_3:T\right]
	\end{align*}    
	The images of $g_\pm$ satisfy the relation
	\[
	p_{13}p_{24} = p_{12}p_{34} + p_{14}p_{23}
	\]
	Due to (\ref{phi_on_y_+}), we see that
	\[
	g_-=g_+\circ \varphi
	\] 
	on the overlap $\trop^{-1}(V_+\cap V_-)$. 
	Thus, gluing them obtains a well-defined analytic map
	\[
	g: X_0^\vee \to Y
	\]
	
	\begin{lem}
		$g$ is injective.
	\end{lem}
	
	\begin{proof}
		Suppose $\mathbf x$ and $\mathbf x'$ are points in $X_0^\vee =\trop^{-1}(V_+)\cup \trop^{-1}(V_-)/\sim_\varphi$ so that $g(\mathbf x)=g(\mathbf x')$.
		We aim to show that $\mathbf x=\mathbf x'$ in $X_0^\vee$.
		If both $\mathbf x$ and $\mathbf x'$ are contained in the image of $\trop^{-1}(V_+)\xhookrightarrow{} X_0^\vee$, then it follows from the fact that $g_+$ is injective.
		Similarly, if both $\mathbf x$ and $\mathbf x'$ are contained in the image of $\trop^{-1}(V_-)\xhookrightarrow{} X_0^\vee$, then it holds as well.
		
		Now, we may assume that $\mathbf x$ is contained in the image of $ \trop^{-1}(V_+ ) \xhookrightarrow{} X_0^\vee$ and that $\mathbf x'$ is contained in the image of $ \trop^{-1}(V_-) \xhookrightarrow{} X_0^\vee$.
		Let $\mathbf y=(y^+_1,y^+_2,y^+_3, y^+)\in\trop^{-1}(V_+)$ and $\mathbf y'=(y_1^-, y_2^-,y_3^-,y^-)\in \trop^{-1}(V_-)$ be their corresponding points respectively.
		Our goal is to show that \(\trop(\mathbf y)\) lies in $\chi_+(U_+\cap U_-)$, that \(\trop(\mathbf y')\) lies in $\chi_-(U_+\cap U_-)$, and that $\varphi(\mathbf y)=\mathbf y'$.

		First, the condition $g(\mathbf x)=g(\mathbf x')$ implies that $g_+(\mathbf y)=g_-(\mathbf y')$.
		Namely, 
		\[\left[\left(\frac{y^+_1y^+_3+T}{y^+}\right):y^+_1y^+_2y^+:y^+_1:y^+_1y^+_2:y^+_1y^+_2y^+_3:T\right] =\left[\frac{T}{y^-}:y^-_1y^-_2y^-(T^{-1}y^-_1y^-_3+1):y^-_1:y^-_1y^-_2:y^-_1y^-_2y^-_3:T\right]
		\] 
		It follows that $y_i^+=y_i^-=:y_i\in\Lambda$ and $y^+=y^- (T^{-1} y_1y_3+1)$.
		Write
		$\trop(\mathbf y)=(x^+_1,x^+_2,x^+_3,x^+)$ and $\trop(\mathbf y')=(x^-_1,x^-_2,x^-_3,x^-)$.
		Then, we must have $x_i^+=x_i^-=:x_i\in\mathbb R$ and
		\begin{equation}
			\label{eq_x+-}
			x^+=x^- + \val(T^{-1}y_1y_3+1) \ge x^- + \min\{x_1+x_3-1, 0\}
		\end{equation}
		Here we use the non-archimedean triangle inequality, so the strict inequality holds only if $x_1+x_3-1=0$.
		Consider the two points $q_+=\chi^{-1}_+(\trop (\mathbf y))\in U_+$ and $q_-=\chi_-^{-1}(\trop(\mathbf y'))\in U_-$. 
		By the definitions of $\chi_\pm$ at \eqref{eq_chi_+}, we may write
		$q_+=(x_1,x_2,x_3,\rho_+)$ and $q_-=(x_1,x_2,x_3,\rho_-)$ for some numbers $\rho_\pm$, where $x^+=\psi_+(x_1,x_2,x_3, \rho_+)$ and $x^-=\psi_-(x_1,x_2,x_3,\rho_-)$.
		By \eqref{eq_psi_defn} and \eqref{eq_x+-}, we must have 
		\begin{equation}
			\label{eq_psi_rho+-}
			\psi(x_1,x_2,x_3,\rho_+) \ge \psi(x_1,x_2,x_3,\rho_-). 
		\end{equation}

		It remains to show $\rho_+=\rho_-$.
		First, using Lemma \ref{lem_monotonic} implies $\rho_+\le \rho_-$.
		If $\rho_+ < \rho_-$, then using Lemma \ref{lem_monotonic} again implies strict inequalities $\psi(x_1,x_2,x_3,\rho_+) > \psi(x_1,x_2,x_3,\rho_-)$ and
		$x^+=x^- + \val(T^{-1}y_1y_3+1) > x^- + \min\{x_1+x_3-1, 0\}$.
		By the property of non-archimedean triangle inequality, we must have $x_1+x_3-1=0$, so $(x_1,x_2,x_3)\in \Pi$ (see \S \ref{s_lag_fib}).
		It follows that $\rho_+\neq 0$ and $\rho_-\neq 0$.
		So, either $0<\rho_+<\rho_-$ or $\rho_+<\rho_-<0$.
		This implies that $\mathbf x$ and $\mathbf x'$ lie in the same chart $\trop^{-1}(V_+)$ or $\trop^{-1}(V_-)$, which was discussed at the beginning of the proof.
	\end{proof}

	Since the dimensions of $X_0^\vee$ and $Y$ are the same, we see that $g$ is an analytic embedding.
	Moreover, one can check that the global analytic functions $W^{i,i+1}$ in \eqref{W_W_ii+1_eq} should correspond to the following global functions on $Y=Gr(2,4)\setminus D_{ac}$:
	\[
	\widetilde W^{12}=\frac{p_{24}}{p_{41}}   \quad , \quad 
	\widetilde W^{23}=\frac{p_{13}}{p_{12}}  \quad , \quad 
	\widetilde W^{34}=\frac{p_{24}}{p_{23}}   \quad , \quad 
	\widetilde W^{41}=T\frac{p_{13}}{p_{34}}  
	\]
	in the sense that $\widetilde W^{i,i+1}\circ g=W^{i,i+1}$, i.e.
	the following diagrams commute
	\[
	\xymatrix{
		X_0^\vee \ar[rr]^g \ar[drr]_{W^{i,i+1}} &  & Y \ar[d]^{\widetilde W^{i,i+1}} \\
		& & \Lambda
	}
	\]

	\smallskip

	This gives the correct Marsh-Rietsch mirror superpotential:
	\begin{equation}
		\label{eq_tilde_W}
		\widetilde W=\frac{p_{24}}{p_{41}}+\frac{p_{13}}{p_{12}}+\frac{p_{24}}{p_{23}}+T\frac{p_{13}}{p_{34}}.
	\end{equation}

	Moreover, the global analytic functions $y_i$ on $X_0^\vee$ (see \eqref{eq_y_i}) correspond to the following global functions on $Y$, still denoted by:
	\[
	y_1  = T\frac{p_{12}}{p_{41}}  \quad , \quad
	y_2 =\frac{p_{23}}{p_{12}} \quad , \quad 
	y_3 =\frac{p_{34}}{p_{23}}
	\]

	\subsubsection{The topological homeomorphism $j$}\label{sec_jmap}
	
	To make the construction more transparent, it is convenient to embed the base space $B$ into a Euclidean space.
	We write $\psi_0(x)=\psi(x,0)$ over $B$. Then define a continuous embedding 
	\begin{equation}
		\label{eq_j}
		j:B\rightarrow \R^5,\quad (x,\r)\mapsto (\theta_0(x,\r),\theta_1(x,\r),\bar x)
	\end{equation}
	where $x=(x_1,x_2,x_3,x_4)$ denotes a point in $\Delta_{2,4}$ and $\bar x=(x_1,x_2,x_3)$ gives a coordinate chart for $B$.
	\begin{align}
		\theta_0( x,\rho)&:=\min\{-\psi( x,\rho),-\psi_0( x)\}+\min\{0,x_1+x_3-1\}\\
		\theta_1( x,\rho)&:=\min\{\psi( x,\rho),\psi_0( x)\}
	\end{align}
	
	To understand the image of $j$ we look at the image of the path \[
	j_x: \ \rho\mapsto  \big( \theta_0(x,\rho), \theta_1(x,\rho) \big)
	\]
	in $\mathbb R^2$ for each fixed $x$.
	Recall that by Lemma \ref{lem_monotonic}, $\psi(x,\rho)$ is decreasing in $\rho$.
	Thus, $\psi(x,\rho)<\psi_0(x)$ for $\rho>0$ and $\psi(x,\rho)>\psi_0(x)$ for $\rho<0$.
	In particular, the trajectory of $j_x$ is a broken line in $\mathbb R^2$ with the corner point 
	\[A(x) =( \min\{0,x_1+x_3-1\} -\psi_0(x) \ , \ \psi_0(x)).
	\]
	
	


	\subsubsection{The fibration map $f$}\label{subsec_f}
	
	We then need to construct an analytic fibration $f$ which fits in the above diagram \eqref{commd}. In the following we denote a point in $Y$ by $z$, and its Pl\"ucker coordinates by $p_{ij}(z)$. 
	
	Consider the following continuous functions on $Y$ (with respect to its analytic topology):
	\begin{align*}
		F_0 &=\min\left\{\val (h_0),\min\left\{0,\val (y_1)+\val (y_3)-1\right\}-\psi_0(\val (y_1),\val (y_2),\val (y_3) ) \right\} \\
		F_1 &=\min\left\{\val (h_1),\psi_0(\val (y_1),\val (y_2),\val (y_3) ) \right\}.
	\end{align*}
	where we put
	\[
	h_0 := \frac{p_{13}}{p_{41}} \qquad , \qquad h_1:= \frac{p_{24}}{p_{23}}
	\]
	Recall that $\val(\cdot )$ denotes the valuation map on the Novikov field $\Lambda$. 
	Then, we define
	\[f:Y\rightarrow \R^5,\quad z\mapsto \left(F_0(z),F_1(z),\val (y_1(z)),\val (y_2(z)),\val (y_3(z))\right)\]
	Note that the Pl\"ucker relation implies that
	\begin{equation} 
		\label{eq_val_Plucker}
		\frac{p_{24}}{p_{23}} + \frac{p_{13}}{p_{41}} = 1+\frac{p_{12}p_{34}}{p_{41}p_{23}} .
	\end{equation}
	Therefore,
	\[
	h_0 h_1 = 1+ T^{-1}y_1y_3 . 
	\]
	In particular,
	\[
	\val(h_0) +\val(h_1) \ge \min \{0, \val(y_1)+\val(y_3) -1\}
	\]
	
	We claim that this is the explicit realization of our mirror fibration.
	First, it is straightforward to check that the diagram \eqref{commd} is commutative.
	Besides, we aim to show that the restriction of $f$ over $j(B_0)$ is an affinoid torus fibration in the sense of Section \ref{s_trop_affinoid_torus_fib}.

	\begin{proof}[Proof of Theorem \ref{main_thm}]
		
		We aim to determine the locus of $f$-smooth points; see Definition \ref{defn_smooth_point}. In fact, it suffices to detect those $f$-smooth points that lie in the embedded image $j(B)$.
		We first show that every point in $j(B_0)$ is $f$-smooth.
		
		Let $q$ be a fixed point in $j(B_0)$ which uniquely corresponds to a point $(\hat x,\hat\rho)=(\hat x_1,\hat x_2,\hat x_3,\hat \rho)$ in $B_0$ via the embedding $j$.
		Recall that $\Pi=\{x_1+x_3-1=0\}$ (see \eqref{eq_Pi})

		If $\hat x\notin \Pi$, then we pick a small neighborhood $U$ on which $x \notin \Pi$ always holds.
		Then, on the analytic domain $f^{-1}(j(U))$, we have $\val(y_1) +\val(y_3)-1\neq 0$.
		Hence, the equality $\val(h_0)+\val(h_1)=\val(1+T^{-1}y_1y_3)=\min\{0, \val(y_1)+\val(y_3)-1\}$ always holds on $f^{-1}(j(U))$.
		Then, eliminating, say, $\val(h_0)$, one obtains
		\begin{align*}
			F_0 = \min \{-\val(h_1), -\psi_0(\val(y_1),\val(y_2),\val(y_3) )\}  + \min \{0,\val(y_1)+\val(y_3)-1\}  
		\end{align*}
		and $F_1 = \min\{\val(h_1), \psi_0(\val(y_1), \val(y_2), \val(y_3))\}$.
		In other words, the fibration map $f=(F_0,F_1,\val(y_1),\val(y_2),\val(y_3)$, restricted on $j(U)$, is essentially determined only by $h_1,y_1,y_2,y_3$, while
		$h_0=h_1^{-1}(1+T^{-1}y_1y_3)$
		does not give an independent function.
		The analytic functions $h_1,y_1, y_2, y_3$ are invertible on $f^{-1}(j(U))$ and then give a commutative diagram
		\[\xymatrix{f^{-1}(j(U))\ar^{(h_1,y_1,y_2,y_3)}[rr]\ar_f[d]&  & \trop^{-1}(V) \ar^{\trop}[d]\\
			j(U) \ar[rr] & & V
		}\]
		for some open subset $V$. Here the bottom horizontal arrow is the composition of $j^{-1}$ and a natural diffeomorphism induced by $\psi$.
		Recall that this diffeomorphism is precisely the one used to define the integral affine coordinates; see \eqref{eq_chi_+} and \eqref{eq_psi_defn}.
		
		If $\hat x\in \Pi$, then $\hat\rho\neq 0$.
		Pick a sufficiently small neighborhood $U$ of $(\hat x,\hat \rho)$ so that $\rho\neq 0$ on $U$.
		Assume $\rho>0$ on $U$.
		Then, by the definition of $j$ \eqref{eq_j}, we see that
		$j(U)$ is of the form
		\[
		\left\{ (-\psi_0(x) + \min \{0,x_1+x_3-1\}  \ , \ \psi(x,\rho) \ , \ x ) \in\mathbb R^5 \mid \rho>0 \ , \ (x,\rho)=(x_1,x_2,x_3,\rho)\in U
		\right\}.
		\]
		Here the first coordinate is fixed. Recall also that $\psi(x,\rho)$ is strictly decreasing in $\rho$.
		Then, the analytic domain $f^{-1}(j(U))$ admits the invertible analytic functions $h_1,y_1,y_2,y_3$ and a similar commutative diagram holds.
		Now, we assume $\rho<0$ on $U$.
		Then, $j(U)$ takes the form
		\[
		\left\{ (-\psi(x,\rho) + \min \{0,x_1+x_3-1\}  \ , \ \psi_0(x) \ , \ x ) \in\mathbb R^5 \mid \rho<0 \ , \ (x,\rho)=(x_1,x_2,x_3,\rho)\in U
		\right\}.
		\]
		Here the second coordinate becomes fixed. One can then check that $h_0,y_1,y_2,y_3$ are invertible analytic functions on $f^{-1}(j(U))$ and a similar commutative diagram holds.
		
		Now, we define 
		\[
		\mathcal Y= f^{-1}(j(B))
		\]
		and
		\[
		\pi^\vee = j^{-1} \circ f |_{\mathcal Y} : \mathcal Y\to B . 
		\]
		
		It follows directly from the construction that the smooth locus of $\pi^\vee$ is exactly the smooth locus of the Lagrangian fibration $\pi$. Furthermore, the integral affine structure induced by $\pi^\vee$ agrees with the one induced by $\pi$.
		
		Regarding the superpotential, the expression computed in \eqref{eq_tilde_W} coincides precisely with the Marsh-Rietsch superpotential \cite{marsh2020b}, upon identifying the Novikov parameter $T$ with the formal variable $q$.
	\end{proof}

	\appendix
	\section{Appendix}
	
	In the appendix, we give a review of basic knowledge of non-archimedean analytic geometry and the Floer-theoretic foundation for the SYZ mirror construction.

	\subsection{Basics of non-archimedean analytic geometry}

	Let $(\Bbbk, |\cdot|)$ be a complete normed field that is algebraically closed. We assume it is \textit{non-archimedean} in the sense that the norm $|\cdot|: \Bbbk\to \mathbb R_{\ge 0}$ satisfies the ultrametric triangle inequality $|x+y|\leqslant \max\{|x|, |y|\}$.
	We can also define the (non-archimedean) valuation $\val: \Bbbk\to \mathbb R\cup \{\infty\}$ by $\val(x)=-\log|x|$, and then $\val(x+y)\geqslant \min\{\val(x),\val(y)\}$.
	Examples of non-archimedean fields include the field $\mathbb C((T))$ of Laurent series, the field of Puiseux series, the field of $p$-adic numbers, and so on.
	For our purposes, the non-archimedean field we focus on is the \textit{Novikov field} defined as follows:
	\[
	\Bbbk=\Lambda= \mathbb C(( T^{\mathbb R})) = \left\{
	\sum_{i=0}^\infty a_i T^{\lambda_i} \mid a_i\in\mathbb C, \lambda_i\in\mathbb R,\lambda_i \nearrow +\infty
	\right\}
	\]
	Its valuation map
	$
	\val: \Lambda\to \mathbb R \cup \{\infty\}
	$
	is defined by sending $x=\sum_{i=0}^\infty a_i T^{\lambda_i}$ with $a_0\neq 0$ to $\val(x) =\lambda_0$.
	The valuation ring is given by $\Lambda_0:= \val^{-1}[0,\infty]=\{x\mid |x| \leqslant 1 \}$, often called the \textit{Novikov ring}.
	It has the maximal ideal $\Lambda_+:=\val^{-1}(0,\infty]=\{x\mid |x|<1\}$.
	The multiplicative group of units in the field $\Lambda$ is denoted as 
	\[
	U_\Lambda:= \val^{-1}(0)=\{x\mid |x|=1\}
	\]
	The standard isomorphism $\mathbb C^*\cong \mathbb C/ 2\pi i \mathbb Z$ naturally extends to $U_\Lambda\cong \Lambda_0/ 2\pi i \mathbb Z$. In particular, for any $y\in U_\Lambda$, there exists some $x\in \Lambda_0$ with $y=\exp(x)$.
	This is a property of the Novikov field.

	\subsubsection{Review of Berkovich geometry}
	
	Let $A$ be a commutative $\Bbbk$-algebra with the structure map $\Bbbk\to A$.
	A \textit{multiplicative seminorm} on $A$ is a map
	$
	\|\cdot\|:A\to \mathbb R_{\ge 0}
	$, extending the norm $|\cdot|$ on $\Bbbk$, such that for all $f,g\in A$ and all $c\in \Bbbk$, we have: (i) $\|0\|=0$ and $\|1\|=1$; (ii) $\|fg\|=\|f\|\cdot \|g\|$; (iii) $\|f+g\|\le \max\{\|f\|,\|g\|\}$.
	We define the \emph{Berkovich spectrum}
	$\mathcal M(A)$ to be the set of multiplicative seminorms on $A$ extending the norm on $\Bbbk$.
	If $x\in\mathcal M(A)$, we often denote by $\|\cdot\|_x: A \to \mathbb R_{\ge 0}$ the corresponding seminorm.
	For each $f\in A$, define the evaluation map
	\[
	\operatorname{ev}_f:\mathcal M(A)\to \mathbb R_{\ge 0},
	\qquad
	x\mapsto \|f\|_x.
	\]
	The \emph{Berkovich topology} on $\mathcal M(A)$ is the coarsest topology for which
	$\operatorname{ev}_f$ is continuous for every $f\in A$.

	Fix $n\ge 0$ and a \textit{polyradius} $r=(r_1,\dots,r_n)\in(\mathbb R_{>0})^n$.
	Define the \emph{Tate algebra of polyradius $r$} (or \emph{polydisk algebra}) by
	\[
	\Bbbk \langle r^{-1}T\rangle
	:=
	\Bbbk \langle r_1^{-1} T_1, \dots, r_n^{-1}T_n \rangle:=
	\left\{
	\sum_{\nu\in\mathbb Z_{\ge 0}^n} a_\nu T^\nu
	\ \middle|\
	a_\nu\in \Bbbk,\ \ |a_\nu|\, r^\nu \to 0 \text{ as }|\nu|\to\infty
	\right\},
	\qquad
	r^\nu:=\prod_{i=1}^n r_i^{\nu_i},
	\]
	equipped with the \emph{Gauss norm}
	\[
	\Bigl\|\sum_\nu a_\nu T^\nu\Bigr\|_{r}
	:=\max_\nu \bigl(|a_\nu|\, r^\nu\bigr).
	\]
	For $r=(1,\dots,1)$, we note that $\Bbbk\langle T_1,\dots,T_n\rangle= \Bbbk \langle r^{-1}T\rangle$ is the usual Tate algebra.
	
	A $\Bbbk$-affinoid algebra is a Banach $\Bbbk$-algebra $A$ for which
	there exist $n$, $r\in(\mathbb R_{>0})^n$, and an ideal $I\subset \Bbbk \langle r^{-1}T\rangle$
	such that $A \cong \Bbbk\langle r^{-1}T\rangle / I $.
	A strictly $\Bbbk$-affinoid algebra is a quotient
	$
	A \cong \Bbbk \langle T_1,\dots,T_n\rangle / I
	$
	for some $n$ and ideal $I\subset \Bbbk\langle T_1,\dots,T_n\rangle$.
	Write $|\Bbbk^\times|\subset \mathbb R_{>0}$ for the value group.
	If $|\Bbbk^\times|=\mathbb R_{>0}$, as is the case for the Novikov field considered here, then every $\Bbbk$-affinoid algebra is strictly $\Bbbk$-affinoid. 
	Therefore, since we work only over the Novikov field, we will henceforth drop the adjective ``strictly'' and also suppress the prefix ``$\Bbbk$-'', and simply use the term \emph{affinoid algebra}.
	
	The Berkovich spectrum $\mathcal M(A)$ of an affinoid algebra $A$ is called an \textit{affinoid space}.

	Starting from affinoid spaces, Berkovich constructs in \cite{Berkovich1993etale}
	the category of \textit{analytic spaces} by a gluing formalism that is somewhat delicate.
	We will not reproduce the construction here; instead, we recall a few properties.
	Let $X$ be a topological space and let $(X_i)_{i\in I}$ be a family of subsets of $X$.
	We say that $(X_i)$ is a \emph{$G$-covering} of $X$ if every point $x\in X$ admits a neighborhood of the form
	$
	U=\bigcup_{i\in J} X_i
	$
	for some \emph{finite} subset $J\subset I$ such that $x\in X_i$ for all $i\in J$.
	Notice that $G$-coverings are not required to be open coverings; rather, they are required to cover neighborhoods of points after passing to finitely many members,
	each of which contains the point.

	Let $X$ be an analytic space in the sense of Berkovich. Then, $X$ is a topological space in which every point admits
	a basis of \emph{compact} neighborhoods.
	The space $X$ has a class of compact subsets called its
	\emph{affinoid domains}. 
	Besides, $X$ also has another class of subsets consisting of the so-called \emph{analytic domains}. By definition, an analytic domain is a subset $V\subset X$ which admits a $G$-covering by
	affinoid domains contained in $V$.
	The $G$-topology on $X$ is defined as follows: its objects are the analytic domains of $X$, its morphisms are inclusions, and its covering families are the $G$-coverings.
	We remark that a compact subset of $X$ which is a finite union of analytic domains of $X$ is an analytic domain.
	
	The spaces most commonly considered in the classical approaches to non-archimedean analytic geometry
	(such as Tate's rigid-analytic spaces \cite{Tate_origin})
	coincide with what Berkovich calls \emph{good} analytic spaces: those for which every point admits
	an affinoid neighborhood, and hence a basis of affinoid neighborhoods.
	Good spaces are often sufficient in practice, but the general category of analytic spaces is sometimes technically convenient.

	\subsubsection{Analytification}
	To any scheme $X$ of finite type, we can
	associate an analytic space $X^{\mathrm{an}}$, called the \emph{analytification} of $X$.
	If $X=\mathrm{Spec} A$ is affine, then the underlying set of $X^{\mathrm{an}}$ is the spectrum $\mathcal M(A)$
	of multiplicative seminorms $\|\cdot\|_x:A\to\mathbb R_{\ge0}$ extending the norm on the ground field $\Bbbk$.
	Alternatively, a point of $X^{\mathrm{an}}$ may be viewed as a pair $(\mathfrak p,|\cdot|_x)$ consisting of a
	prime $\mathfrak p\subset A$ together with an absolute value on the residue field $\kappa(\mathfrak p)$ extending that of $\Bbbk$.
	For general $X$, choose an affine open cover $X=\bigcup_i U_i$ with $U_i=\mathrm{Spec} A_i$ and set
	$U_i^{\mathrm{an}}:=\mathcal M(A_i)$.
	On overlaps $U_i\cap U_j$, these analytifications identify with common
	analytic domains, and one defines $X^{\mathrm{an}}$ by gluing the spaces $U_i^{\mathrm{an}}$ along these domains.
	The resulting space is an analytic space and is good in the above sense.
	

	\subsubsection{Tropicalization map and affinoid torus fibration}
	\label{s_trop_affinoid_torus_fib}
	
	Recall that we focus on the case of the Novikov field $\Bbbk=\Lambda$. 
	Let $\mathbf T=\mathrm{Spec} \ \Lambda[T_1^{\pm1},\dots,T_n^{\pm1}]$ be the algebraic torus, and let $\mathbf T^{\mathrm{an}}$ be its
	Berkovich analytification. There is a natural \emph{tropicalization} map
	\[
	\trop:\mathbf T^{\mathrm{an}}\longrightarrow \mathbb R^n,
	\qquad
	x\longmapsto \bigl(-\log|T_1(x)|,\dots,-\log|T_n(x)|\bigr),
	\]
	where $|T_i(x)|=\|T_i\|_{x}$ denotes the value of the coordinate function $T_i$ under the multiplicative seminorm
	corresponding to $x$. 
	For a point $x=(a_1,\dots,a_n)$ in $(\Lambda^*)^n$, one has
	$\trop(x)=(\val(a_1),\dots,\val(a_n))$.
	Roughly speaking, $\trop$ records the non-archimedean ``size'' of the coordinates.
	Abusing the notation, we often simply write $(\Lambda^*)^n$ for $\mathbf T^{\mathrm{an}}$.
	
	The tropicalization map is known to be continuous, so given a Euclidean open subset $U$ of $\mathbb R^n$, the preimage $\trop^{-1}(U)$ is an open subset of $(\Lambda^*)^n$ and thus an analytic domain of $(\Lambda^*)^n$.
	
	Let $\Delta\subset\mathbb R^n$ be a rational convex polyhedron, given by finitely many affine-linear
	inequalities
	$
	\sum_{j=1}^n b_{ij}u_j \ge c_i$, $ i=1,\dots,N$
	with $b_{ij}\in\mathbb Z$ and $c_i\in\mathbb R$.
	A result of Einsiedler-Kapranov-Lind \cite[Proposition~3.1.5]{EKL} implies that
	$\trop^{-1}(\Delta)\subset (\Lambda^*)^n$ is an affinoid domain (not just an analytic domain). Moreover, the corresponding affinoid algebra is
	\[
	\Lambda\langle \Delta\rangle
	=
	\left\{
	\sum_{i=0}^{\infty} s_i\,Y^{\alpha_i}\in \Lambda[[\mathbb Z^n]]
	\ \middle|\
	s_i\in\Lambda,\ \alpha_i\in\mathbb Z^n,\ \val(s_i)+\alpha_i\cdot\gamma \to +\infty
	\text{ for all }\gamma\in\Delta
	\right\},
	\]
	where $\Lambda[[\mathbb Z^n]]\cong \Lambda[[Y_1^{\pm 1},\dots,Y_n^{\pm 1}]]$ is the Laurent formal power
	series ring; see also \cite[6.1.4]{BGR}.
	For $f_1,\dots,f_r\in \Lambda\langle\Delta\rangle$, the common zero locus of
	$f_1,\dots,f_r$ inside $\trop^{-1}(\Delta)$ is again an affinoid domain and can be identified with
	$
	\mathcal M\bigl(\Lambda\langle\Delta\rangle/(f_1,\dots,f_r)\bigr)
	$; cf.\ \cite[Proposition~3.1.8]{EKL}.

	Let $X_0$ be an analytic space and $B_0$ a topological space.

	\begin{defn}
		\label{defn_smooth_point}
		Fix a continuous map $f:X_0\to B_0$.
		A point $b\in B_0$ is called \textit{$f$-smooth} if there is
		an open neighborhood $U\subset B_0$, an open subset
		$V\subset \mathbb R^n$, and isomorphisms
		\[
		\chi :U\xrightarrow{\ \sim\ }V,
		\qquad
		\varphi:f^{-1}(U)\xrightarrow{\ \sim\ }\trop^{-1}(V)
		\]
		such that $\chi \circ f = \trop\circ \varphi$.  
		The map $f$ is called an \textit{affinoid torus fibration}
		if every point $b\in B_0$ is $f$-smooth. cf. \cite[Section 4]{KSAffine}.
	\end{defn}

	We remark that an open subset of an analytic space is an analytic domain, and any analytic domain inherits a canonical structure of an analytic space.
	Although $\trop^{-1}(V)$ need not be affinoid when $V$ is open, it is an analytic domain
	covered by affinoid domains of the form $\trop^{-1}(\Delta)$ with $\Delta\subset V$ a rational
	convex polyhedron. Thus, one may recover affinoid local models by restricting to such relatively
	compact polyhedra inside $V$.

	\subsection{Family Floer SYZ mirror construction}
	\label{s_appendix_family_floer}
	
	Let $(X, \omega)$ be a symplectic manifold of real dimension $2n$. Suppose there is a Lagrangian torus fibration $\pi_0:X_0\to B_0$ on some open domain $X_0\subset X$.
	Let $J$ be an almost complex structure such that there is no $J$-holomorphic stable disk of negative Maslov index bounding a Lagrangian fiber.
	This is indeed the case for the Lagrangian fibration considered in this paper.
	
	\subsubsection{General statement}
	The family Floer mirror construction developed in \cite{Yuan_I_FamilyFloer} states that 
	
	\begin{thm}
		\label{family_floer_thm}
		Given $(X,\pi_0)$ as above, there is an analytic space $X_0^\vee$ over the Novikov field, an analytic function $W:X_0^\vee \to \Lambda$, and an affinoid torus fibration $\pi_0^\vee:X_0^\vee \to B_0$ such that
		\begin{itemize}
			\itemsep 0pt 
			\item they are unique up to isomorphism of analytic spaces
			\item the integral affine structure on $B_0$ induced by $\pi_0^\vee$ is identical to the one induced by $\pi_0$
			\item the set of rigid points in $X_0^\vee$ coincides with
			\[
			\bigcup_{q\in B_0} H^1(L_q; U_\Lambda)
			\]
		\end{itemize}
	\end{thm}

	Let's describe the Floer-theoretic mechanism underlying the above theorem.
	For clarity, we choose and fix a Riemannian metric on $B_0$.
	Let $U\subset B_0$ be a contractible open subset, and let $q_0\in B_0$ be a specified point such that the diameter of $U\cup \{q_0\}$ is sufficiently small.
	We allow $q_0 \notin U$.
	Remark that the point $q_0$ is introduced primarily for theoretical purposes, and in practice we may often omit mentioning it.
	
	Let $\chi=(\chi_1,\dots, \chi_n):(U, q_0) \to (V,c) \subset \mathbb R^n$ be a (pointed) integral affine coordinate chart with $\chi(q_0)=c$. Then, in the context of the theorem, there is an isomorphism of analytic spaces
	\[
	\tau: (\pi_0^\vee)^{-1}(U) \xrightarrow{\cong} \trop^{-1}(V - c)
	\]
	with $\trop \circ \tau = \chi \circ \pi_0^\vee$. Here $V-c$ consists of the points of the form $v-c$ where $v\in V$.
	Concretely, let $\mathbf y$ be a rigid point in $H^1(L_q; U_\Lambda)$ with $q\in U$. Denote the natural pairing $\pi_1(L_q)\times H^1(L_q; U_\Lambda) \to U_\Lambda$ by $(\alpha, \mathbf y)\mapsto \mathbf y^{\alpha}$.
	Suppose $\{e_i=e_i(q) \mid i=1,\dots, n \}_{q\in U}$ is the $\mathbb Z$-basis of $\pi_1(L_q)$ corresponding to the integral affine chart $\chi$.
	Then,
	\[
	\tau(\mathbf y) = (T^{\chi_1(q)} \mathbf y^{e_1(q)} , \dots, T^{\chi_n(q)} \mathbf y^{e_n(q)} )
	\]
	
	Regarding the superpotential function $W: X_0^\vee \to \Lambda$, we have
	$
	W_\tau:= W\circ \tau^{-1} : \trop^{-1}(V-c) \to \Lambda
	$
	described by
	\[
	(y_1,\dots, y_n) \mapsto  \sum_{\beta\in \pi_2(X, L_{q_0}) \ , \  \mu(\beta)=2} T^{E(\beta)} y_1^{\partial\beta \cap e_1} \cdots y_n^{\partial\beta\cap e_n} \mathsf n_{\beta}
	\]
	where $\mu(\beta)$ is the Maslov index of $\beta$, $E(\beta)=\int_\omega\beta$, and $\mathsf n_{\beta}$ is the virtual count of pseudo-holomorphic stable disk in class $\beta$.
	
	Let's next describe the gluing among these charts. Suppose $\chi':(U', q_0')\to (V',c')$, $\tau'$, and $\{e_i'(q)\}$ are given as above. Replacing $U$ and $U'$ by $U\cap U'$ if necessary, we may assume $U=U'$.
	The transition map $\Phi=\tau'\circ \tau^{-1} : \trop^{-1}(V-c)\to \trop^{-1}(V' -c')$ is an analytic isomorphism that covers the integral affine transition map $\upchi:=\chi'\circ \chi^{-1}:V-c\to V'-c'$.
	If $\upchi$ is of the form $\mathbf x\mapsto A\mathbf x + \mathbf b$ with $A=(a_{ij})\in GL(n,\mathbb Z)$ and $\mathbf b =(b_j) \in\mathbb R^n$, then it is covered by the analytic isomorphism 
	\[
	\trop^{-1}(V-c)\to \trop^{-1}(V'-c') , \qquad  y_j \mapsto T^{b_j}\cdot \prod_k y_k^{a_{jk}}
	\]
	Therefore, without loss of generality, we may assume $\upchi$ is the identity map. The transition map $\Phi$ is of the form
	\[
	\Phi: \trop^{-1}(V-c)\to \trop^{-1}(V'-c') , \qquad 
	y_j\mapsto y_j F_j(y_1,\dots, y_n)
	\]
	where $F_j$ is a convergent formal power series determined by the $A_\infty$ algebra structures associated to the two Lagrangian fibers $L_{q_0}$ and $L_{q_0'}$.
	Moreover, $F_j$ is only contributed by the presence of Maslov-0 pseudo-holomorphic disks bounded by some adjacent Lagrangian fiber.
	We also refer to \cite[Section 4]{Yuan_local_SYZ} for a coordinate-free description of Theorem \ref{family_floer_thm}.
	While $F_j$'s can be quite complicated in general, it is eventually proved in \cite{Yuan_I_FamilyFloer} that these local superpotential functions $W_\tau$ are compatible with the transition maps, that is, for the transition map $\Phi=\tau'\circ \tau^{-1}$ between the two charts $\tau$ and $\tau'$, we have 
	\[
	W_{\tau'} \circ \Phi =W_\tau
	\]

	\subsubsection{Void wall-crossing}
	Let $B_1\subset B_0$ be a contractible open set, and let
	$B_2$ be a small contractible neighborhood of $B_1$ in $B_0$. One can find $\epsilon>0$ sufficiently small so that the reverse isoperimetric constant for any Lagrangian fiber over $B_1$ is uniformly larger than $\epsilon$; see \cite{Yuan_I_FamilyFloer}.
	Fix an integral affine coordinate chart $\chi:B_2\to \R^n$. Assume that for every $q\in B_1$, the Lagrangian fiber $L_q$ bounds no non-constant Maslov index $0$ holomorphic disk. Then $(\pi_0^\vee)^{-1}(B_2)$ admits a single analytic chart, namely an identification
	\[
	(\pi_0^\vee)^{-1}(B_2)\ \cong\ \trop^{-1}\bigl(\chi(B_2)\bigr).
	\]
	Indeed, since $B_2$ is contractible, we fix a pointed affine chart
	$
	\chi:(B_2,q_0)\to (V,c)
	$.
	Cover $B_2$ by open sets $U_i$ of diameter $<\epsilon$ with basepoints $q_i\in B_1$ and pointed charts
	$
	\chi_i=\chi|_{U_i}:(U_i,q_i)\to (V_i,c_i) $ for $i\in I$.
	For each $U_i$, we have
	$
	(\pi_0^\vee)^{-1}(U_i)\ \cong\ \trop^{-1}(V_i-c_i)
	$
	as above. The assumption that there are no Maslov-$0$ holomorphic disks along the fibers over $B_1$ implies that the transition maps have no twisting terms and reduce to simple rescalings
	$
	y \mapsto T^{c}y
	$.
	If we glue various $\trop^{-1}(V_i-c_i)$ through these simple rescalings, then the outcome can be identified with $\trop^{-1}(\chi(B_2))$ as desired.
	See \cite[Proposition 4.4]{Yuan_local_SYZ}.
	
	\bibliographystyle{abbrv}
	\bibliography{mybib_g24}		

\begin{thebibliography}{10}

\bibitem{AuTDual}
D.~Auroux.
\newblock {Mirror symmetry and T-duality in the complement of an anticanonical
  divisor}.
\newblock {\em Journal of G{\"o}kova Geometry Topology}, 1:51--91, 2007.

\bibitem{Berkovich1993etale}
V.~G. Berkovich.
\newblock {\'E}tale cohomology for non-archimedean analytic spaces.
\newblock {\em Publications Math{\'e}matiques de l'IH{\'E}S}, 78:5--161, 1993.

\bibitem{Berkovich_2012spectral}
V.~G. Berkovich.
\newblock {\em Spectral theory and analytic geometry over non-Archimedean
  fields}.
\newblock Number~33. American Mathematical Soc., 2012.

\bibitem{BGR}
S.~Bosch, U.~G{\"u}ntzer, and R.~Remmert.
\newblock {\em Non-Archimedean analysis, volume 261 of Grundlehren der
  Mathematischen Wissenschaften}.
\newblock Springer-Verlag, Berlin, 1984.

\bibitem{action_angle}
J.~J. Duistermaat.
\newblock On global action-angle coordinates.
\newblock {\em Communications on pure and applied mathematics}, 33(6):687--706,
  1980.

\bibitem{EguchiHoriXiong1997}
T.~Eguchi, K.~Hori, and C.-S. Xiong.
\newblock Quantum cohomology and virasoro algebra.
\newblock {\em Physics Letters B}, 402(1--2):71--80, 1997.

\bibitem{EKL}
M.~Einsiedler, M.~Kapranov, and D.~Lind.
\newblock Non-archimedean amoebas and tropical varieties.
\newblock {\em Journal f{\"u}r die reine und angewandte Mathematik (Crelles
  Journal)}, 2006(601):139--157, 2006.

\bibitem{hong2023immersed}
H.~Hong, Y.~Kim, and S.-C. Lau.
\newblock Immersed two-spheres and {SYZ} with application to {Grassmannians}.
\newblock {\em Journal of Differential Geometry}, 125(3):427--507, 2023.

\bibitem{hori2002mirror}
K.~Hori.
\newblock Mirror symmetry and quantum geometry.
\newblock {\em Proceedings of the International Congress of Mathematicians},
  Vol. III (Beijing, 2002):431--443, 2002.

\bibitem{KSAffine}
M.~Kontsevich and Y.~Soibelman.
\newblock Affine structures and non-archimedean analytic spaces.
\newblock In {\em The unity of mathematics}, pages 321--385. Springer, 2006.

\bibitem{marsh2020b}
B.~Marsh and K.~Rietsch.
\newblock {The B-model connection and mirror symmetry for Grassmannians}.
\newblock {\em Advances in Mathematics}, 366:107027, 2020.

\bibitem{nishinou2010toric}
T.~Nishinou, Y.~Nohara, and K.~Ueda.
\newblock {Toric degenerations of Gelfand--Cetlin systems and potential
  functions}.
\newblock {\em Advances in Mathematics}, 224(2):648--706, 2010.

\bibitem{rietsch2008mirror}
K.~Rietsch.
\newblock {A mirror symmetric construction of $qH_T^*(G/P)_{(q)}$ }.
\newblock {\em Advances in Mathematics}, 217(6):2401--2442, 2008.

\bibitem{SYZ}
A.~Strominger, S.-T. Yau, and E.~Zaslow.
\newblock {Mirror symmetry is T-duality}.
\newblock {\em {Nuclear Physics. B}}, 479(1-2):243--259, 1996.

\bibitem{Tate_origin}
J.~Tate.
\newblock Rigid analytic spaces.
\newblock {\em Inventiones mathematicae}, 12(4):257--289, 1971.

\bibitem{Yuan_conifold}
H.~Yuan.
\newblock {Family Floer SYZ singularities for the conifold transition}.
\newblock {\em Accepted by Kyoto Journal of Mathematics}.

\bibitem{Yuan_I_FamilyFloer}
H.~Yuan.
\newblock {Family Floer program and non-archimedean SYZ mirror construction}.
\newblock {\em arXiv preprint arXiv: 2003.06106}, 2020.

\bibitem{Yuan_A_n}
H.~Yuan.
\newblock {Family Floer SYZ conjecture for $A_n$ singularity}.
\newblock {\em arXiv preprint arXiv:2305.13554}, 2023.

\bibitem{Yuan_local_SYZ}
H.~Yuan.
\newblock {Family Floer mirror space for local SYZ singularities}.
\newblock {\em Forum of Mathematics, Sigma}, 12:e119, 2024.

\bibitem{Yuan_c_1}
H.~Yuan.
\newblock {Family Floer superpotential’s critical values are eigenvalues of
  quantum product by $c_1$}.
\newblock {\em Selecta Mathematica}, 31(1):13, 2025.

\end{thebibliography}
	\addtocontents{toc}{\protect\setcounter{tocdepth}{2}}
	
\end{document}